\documentclass[11pt]{article}

\usepackage{amsfonts}
\usepackage{amsmath}
\usepackage[linesnumbered,lined,boxed,commentsnumbered]{algorithm2e}
\usepackage{amssymb,color}
\usepackage{epsfig}
\usepackage{bm}
\usepackage{enumerate}
\numberwithin{equation}{section} 

\usepackage[normalem]{ulem}

\newcommand{\rhs}{f^\text{\tiny inc}}
\DeclareMathOperator\sign{sign}

\title{
Domain decomposition for quasi-periodic scattering by layered media via robust boundary-integral equations at all frequencies
}
\author{Carlos P\'erez-Arancibia\footnote{Department of Mathematics,
    MIT, Cambridge, MA 02139, and Institute for Mathematical and Computational Engineering, School of Engineering and Faculty of Mathematics, Pontificia Universidad Católica de Chile, Santiago, Chile.  Email: cperezar@mit.edu,},
  \, Stephen P. Shipman\footnote{Dept. of Mathematics, Louisiana State
    University, Baton Rouge, LA \ 70803. Email:
    shipman@math.lsu.edu},
  \, Catalin Turc\footnote{Dept. of Math.
    Sciences, New Jersey Inst. of Technology, Newark, NJ 07102. Email: catalin.c.turc@njit.edu},
    \, Stephanos
  Venakides\footnote{Dept. of Mathematics, Duke University, Durham, NC
    \ 27708. Email: ven@math.duke.edu} }

\newtheorem{theorem}{Theorem}[section]

\newtheorem{corollary}[theorem]{Corollary}
\newtheorem{remark}[theorem]{Remark}
\newenvironment{proof}{\hspace{0.5cm} {\bf Proof.}}
{$\quad {}_\blacksquare$\vspace{0.3cm}}

\newcommand{\notesps}[1]{#1}

\date{}
\newcommand{\triple}[1]{{\left\vert\kern-0.25ex\left\vert\kern-0.25ex\left\vert #1 
    \right\vert\kern-0.25ex\right\vert\kern-0.25ex\right\vert}}

\begin{document}
\maketitle
\begin{abstract}
  We develop a non-overlapping domain decomposition method (DDM) for scalar wave scattering by periodic layered media.  Our approach relies on robust boundary-integral equation formulations of Robin-to-Robin (RtR) maps throughout the frequency spectrum, including cutoff (or Wood) frequencies.   We overcome the obstacle of non-convergent quasi-periodic Green functions at  these frequencies by incorporating newly introduced shifted Green functions. Using the latter in the definition of quasi-periodic boundary-integral operators leads to rigorously stable computations of RtR operators. We develop Nystr\"om discretizations of the RtR maps that rely on trigonometric interpolation, singularity resolution, and fast convergent windowed quasi-periodic Green functions. We solve the tridiagonal DDM system via recursive Schur complements and establish rigorously that this procedure is always completed successfully. We present a variety of numerical results concerning Wood frequencies in two and three dimensions as well as large numbers of layers.

  \textbf{Keywords}: Helmholtz transmission problem, domain decomposition, periodic layered media, lattice sum.\\
   
 \textbf{AMS subject classifications}: 
 65N38, 35J05, 65T40,65F08
\end{abstract}

\section{Introduction}
\label{intro}

Simulation of electromagnetic wave propagation in periodic layered media has numerous applications in optics and photonics (photovoltaic devices, computation of plasmons, {\itshape etc.}). The use of periodic structures, such as diffraction gratings, which transmit and reflect waves along a discrete set of propagating directions,  opens up interesting possibilities to guide and direct waves in  unusual ways. Volumetric discretizations (finite-difference (FD)~\cite{taflove2005computational}, finite element (FE)~\cite{jin2015finite}), that constitute the vast majority of numerical methods, require very large numbers of unknowns to suppress their inherent pollution effect, and thus produce very large linear systems requiring good preconditioners, which may not be readily available.  Furthermore, such methods must enforce radiation conditions in infinite domains by means of absorbing boundary conditions (ABC) or perfectly matched layers (PML) \notesps{(see, for example, \cite{Berenger1994,GivoliKeller1990,HanWu1985})}, both of which meet difficulties in the treatment of surface waves and evanescent modes~\cite{johnson2008notes}.

In the technologically relevant case of piecewise constant periodic layered media,  simulation methods based on boundary-integral equations (BIE) and  quasi-periodic Green functions are attractive candidates.  Radiation conditions are enforced automatically, and  discretizations of material interfaces are much smaller than volumetric discretizations and do not suffer from the pollution effect. 
Quasi-periodic Green functions are infinite sums of free-space Green functions with periodically distributed monopole singularities.  \notesps{These double sums converge, although very slowly, for all but a discrete set of ``cutoff" frequencies, for a given quasi-periodicity parameters (Bloch wavevector).  These are cutoff frequencies at which a Rayleigh diffraction mode transitions between propagating and evanescent and the number of propagating directions jumps.  Around these frequencies, the energy is rapidly redistributed along emerging new directions and is associated with anomalous scattering behavior.  These frequencies are often referred to as Wood frequencies (or Wood configurations of wavevector and frequency) because their problematic association in the literature to Wood's anomaly; see the works \cite{HesselOliner1965,NichollsOhJohnson2016,StewartGallaway1962,Wood1935},\cite[Ch.~1]{EnochBonod2012} and references therein for discussions on this phenomenon.}
Popular methods for accelerating the slow convergence at non-Wood frequencies include Ewald summation~\cite{Ewald} and lattice sums~\cite{Linton}.  \notesps{At very high frequencies, asymptotic methods help to accelerate computation; see for example~\cite{KurkcuReitich2009}.}

While the underlying scattering problems are, with regard to the PDE, generically stable at Wood configurations of wavevector and frequency, the latter pose a challenge to BIE for quasi-periodic problems.  In three dimensions, they become increasingly close together at high frequency, and this puts the solution of quasi-periodic problems based on the quasi-periodic Green function out of reach.  For periodic layered media with large numbers of layers, such as thin films used in photovoltaic cells, the probability of encountering Wood frequencies is high.  Another difficulty is the need for an efficient algorithm for the evaluation of quasi-periodic Green functions and their integration into existing fast BIE solvers.  In the solution of the ensuing dense linear systems, the BIE formulations of periodic layered media give rise to tridiagonal solvers, whose structure can be exploited  to lead to efficient direct solvers~\cite{cho2015robust}.

The above challenges faced by BIE-based quasi-periodic solvers were addressed
in two  recent computational methods. Alternative periodization schemes for boundary-integral formulations of quasi-periodic problems that do not rely on the classical quasi-periodic Green functions were proposed in~\cite{cho2015robust,Barnett2011}. These methods have ideas in common with the work presented in~\cite{gumerov2014method} as well as kernel-idependent FMM methods~\cite{YING2004591} and rely on representations of fields as sums of layer potentials and linear combinations of free-space fundamental solutions (radial basis functions)~\cite{gumerov2014method} whereby the quasi-periodicity and radiations conditions are enforced numerically and are not intrinsically satisfied.   This approach gives rise to efficient direct solvers for transmission problems in two-dimensional periodic layered media, and can yield results even at Wood frequencies~\cite{cho2015robust,LaiKobayashiBarnett2015}.  Its rigorous analysis appears to be absent in the literature, 
and we are not aware of evidence that these methods are capable of handling Wood frequencies in three dimensions.  

For the first time in this arena, a rigorous solution was provided to the problem of boundary-integral equation formulations of quasi-periodic problems at Wood frequencies in both two and three dimensions through a new method, in which the well-posedness of the formulation and the stability of the numerical scheme were proven, each in its own right~\cite{bruno2016superalgebraically,Delourme,bruno2017three}.   Smooth windowed truncations of the lattice sums for the Green functions were introduced and analyzed in~\cite{bruno2016superalgebraically}. It was first shown in the same reference~\cite{bruno2016superalgebraically} that the windowed Green functions (WGF) converge to their corresponding quasi-periodic Green functions superalgebraically away from Wood frequency/wavevector configurations as the radius of truncation increases.  The incorporation of WGF in existing fast boundary-integral solvers is relatively straightforward. Remarkably, the WGF method can be adapted to handle scattering problems in layered media whose infinite interfaces are no longer periodic~\cite{bruno2016windowed,bruno2017windowed}. Then, shifted Green functions that converge at and around Wood anomalies were used in a boundary-integral equation setting to provide accurate solutions of scattering problems for perfectly reflecting gratings throughout the frequency spectrum. The shifted Green functions converge algebraically fast at Wood frequencies, and the rate of convergence grows with the number of shifts. However, the shifted Green functions introduce new singularities (poles) in addition to those already present in the quasi-periodic Green functions. Remarkably, these additional singularities turn out to be benign in the case of perfectly reflecting periodic gratings as they can be arranged to be outside the computational domain if one uses indirect formulations~\cite{Delourme,bruno2017three}.

In this article, we extend the shifted Green function method to the case of scalar transmission problems in periodic layered media.  By using a domain decomposition method (DDM), we overcome the difficulty of poles of the shifted Green function inside the computational domain, and we establish the well-posedness of the ensuing system of boundary-integral equations.  \notesps{There is a vast literature on DDM; the reader is referred to the seminal works of B. Despr\'es~\cite{Depres,Despres1991} and the expository books~\cite{DoleanJolivetNataf2015,ToselliWidlund2006}.
DDM is well suited to the Helmholtz/Maxwell equations in periodic layered media because of the robustness of the Robin-to-Robin (RtR) operator for each layer~\cite{schadle2007domain,nicholls2018stable}.}  For a given periodic layer, it maps incoming (interior boundary) Robin data to outgoing (exterior boundary) Robin data on the interfaces that bound that layer.  In this way, Robin data are matched on each interface of material discontinuity.  This procedure produces a tridiagonal system whose unknowns are the Robin data on interfaces, and whose non-zero blocks consist of RtR operators.  If a particular layer has constant material properties, the RtR operators can be computed robustly in terms of boundary-integral operators that use  the ordinary quasi-periodic Green function for frequencies that are not Wood frequencies, and shifted Green functions for wavenumbers that are near or at Wood frequencies. Interestingly, the computations of RtR do not require use of hypersingular boundary-integral operators. We establish rigorously in this work two important~facts.

\begin{enumerate}
\item The computations of RtR maps via boundary-integral operators are robust throughout the frequency spectrum if shifted Green functions are employed at Wood frequencies.
  \item The DDM for solution of scalar transmission problems in periodic layered media with piecewise constant material properties presented in this paper is equivalent to the original PDE, assuming that the PDE problem is well-posed.
  \end{enumerate}

We develop a high-order discretization of the tridiagonal DDM system based  on Nystr\"om discretizations of periodic boundary-integral operators. The latter, in turn, rely on trigonometric interpolation, logarithmic singularity extraction in two dimensions and analytic resolution of singularity in three dimensions, and the windowed Green function method~\cite{Delourme,bruno2017three}. We solve the DDM system using recursive Schur complements to eliminate sequentially the discretized Robin data corresponding to each layer in a top-down sweep, a procedure that leads to a computational cost that is linear in the number of layers. We also present theoretical arguments to explain why the Schur complement elimination procedure can be always completed successfully. The variety of two- and three-dimensional numerical results presented in this paper showcase the capability of our DDM solver to handle large numbers of layers, challenging Wood configurations, and inclusions in a periodic layered medium.  The DDM solvers presented in this paper, being built on quasi-periodic Green functions, must be re-assembled when the quasi-periodic parameter changes. Also, the computations of RtR operators require inversions of boundary-integral operators. In summary, the DDM solvers developed in this paper enjoy the following attractive features.
\begin{itemize}
\item The computations of the RtR maps are stable across the frequency spectrum.
\item The DDM system can be solved via recursive Schur complements, leading to a computational cost and memory usage that are linear in the number of layers; and it can be shown rigorously that this procedure does not break down.
\item The DDM approach, being modular, allows for use of heterogeneous discretizations such as FE and BIE and use of non-conforming discretizations on interfaces pertaining to layers with different material properties.
  \item DDM are easily parallelizable.
  \end{itemize}
 The integration of WGF and shifted Green functions  in existing three-dimensional boundary-integral operator discretizations presented in this contribution is relatively seamless and results in a rigorous treatment of Wood configurations in three dimensions.
DDM approaches will be feasible for the solution of three-dimensional electromagnetic transmission problems in periodic layered media based  on quasi-optimal transmission conditions~\cite{jerez2017multitrace,boubendir2017domain,boubendir2014well} that renders them amenable to Krylov subspace iterative solvers.  \notesps{Quasi-optimal transmission conditions arise from a judicious choice of the complex wavenumber in the transmission operator that gives rise to a DDM whose rate of convergence is practically independent of frequency~\cite{boubendirDDM}.}

\notesps{The RtR DDM that employs a shifted Green-function scheme can handle interfaces between layers that are very general (including those that are not the graph of a function) and general frequencies.
There are of course situations in which other methods would be superior or should be used in combination with the RtR DDM.  In the case of small, smooth perturbations of flat interfaces, the method of variation of boundaries would provide increased acceleration~\cite{BrunoReitich1993}, even if the perturbations are not that small~\cite{NichollsReitich2001b,nicholls2018stable}.}
\notesps{And as noted above, at high frequencies, asymptotic methods should be used to accelerate the computation~\cite{KurkcuReitich2009}.}

The paper is organized as follows. In Section~\ref{MS10} we present the scalar scattering problem in two-dimensional layered media and we review the main results about the well-posedness of these problems.  In Section~\ref{DDM} we present a DDM formulation of the transmission problems that uses matching of classical Robin boundary conditions of the material interfaces, and we present computations of ensuing RtR maps that are shown to be stable throughout the frequency spectrum.  We continue in Section~\ref{nystrom} with a description of the Nystr\"om discretization of the RtR maps and we provide and analyze a recursive Schur complement elimination algorithm for the direct solution of the discrete DDM system.  Finally, we present in Section~\ref{num} a variety of numerical results of wave scattering at mostly Wood frequency configurations in periodic layered media.

\parskip 2pt plus2pt minus1pt

\section{Scalar transmission problems \label{MS10}}
We consider the problem of quasi-periodic scattering by penetrable homogeneous periodic layers. For the sake of simpler notations, we present the two-dimensional case. We mention that all the derivations that we present are easily translatable to three-dimensional configurations. The periodicity of the layers is taken to be in the horizontal $x_1$ direction, that is the layers are given by $\Omega_j=\{(x_1,x_2)\in\mathbb{R}^2: F_{j}(x_1)\leq x_2\leq F_{j-1}(x_1)\}$ for $0<j<N$ and $\Omega_0=\{(x_1,x_2)\in\mathbb{R}^2:F_0(x_1)\leq x_2\}$ and $\Omega_{N+1}=\{(x_1,x_2)\in\mathbb{R}^2:x_2\leq F_N(x_1)\}$, and all the functions $F_j$ are periodic with principal period $d$, that is $F_j(x_1+d)=F_j(x_1)$ for all $0\leq j\leq N$. We assume that the medium occupying the layer $\Omega_j$ is homogeneous and its permitivity is $\varepsilon_j$; the wavenumber $k_j$ in the layer $\Omega_j$ is given by $k_j=\omega\sqrt{\varepsilon_j}$. A plane wave $u^\text{\tiny inc}(\mathbf{x})=\exp(i(\alpha x_1+i\beta x_2))$ where $\alpha^2+\beta^2=k_0^2$ impinges on the layered structure. We seek $\alpha$-quasi-periodic fields $u_j$ (i.e. $u_j(x_1+d,x_2)=e^{i\alpha d}u(x_1,x_2)$ for all $(x_1,x_2)\in\mathbb{R}^2$) that satisfy the following system of equations:
\begin{equation}\label{system_t}
\begin{array}{rclll}
 \Delta u_j +k_j^2 u_j &=&0& {\rm in}& \Omega_j^{per}:=\{(x_1,x_2)\in\Omega_j: 0\leq x_1\leq d\},\smallskip\\
  u_j+\delta_0u^\text{\tiny inc}&=&u_{j+1}&{\rm on}& \Gamma_j=\{(x_1,x_2):0\leq x_1\leq d,\ x_2=F_j(x_1)\},\smallskip\\
  \gamma_j(\partial_{n_j}u_j+\delta_0\partial_{n_j}u^\text{\tiny inc})&=&-\gamma_{j+1}\partial_{n_{j+1}}u_{j+1}&{\rm on}& \Gamma_j,\smallskip
\end{array}\end{equation}
where $\delta_0$ is the Dirac distribution supported on $\Gamma_0$ and $n_j$ denote the unit normals to the boundary $\partial\Omega_j$ pointing to the exterior of the subdomain $\Omega_j$. Note that we assigned to the partial derivatives on a given interface the index of the domain on whose side the partial derivative is taken; thus, on the interface $\Gamma_j$ we have $n_j=-n_{j+1}$. We also assume that $u_0$ and $u_{N}$ in equations~\eqref{system_t} are radiative in $\Omega_0$ and $\Omega_{N+1}$ respectively. The latter requirement amounts to expressing the solutions $u_0$ and $u_{N+1}$ in terms of Rayleigh series
\begin{equation}\label{eq:rad_up}
  u_0(x_1,x_2)=\sum_{r\in\mathbb{Z}} C_r^{+}e^{i\alpha_r x_1+i\beta_{0,r} x_2},\quad x_2>\max{F_0}
\end{equation}
and
\begin{equation}\label{eq:rad_down}
  u_{N+1}(x_1,x_2)=\sum_{r\in\mathbb{Z}} C_r^{-}e^{i\alpha_r x_1-i\beta_{N+1,r} x_2},\quad x_2<\min{F_N}
\end{equation}
in which $\alpha_r=\alpha+\frac{2\pi}{d}r$ and $\beta_{0,r}=(k_0^2-\alpha_r^2)^{1/2}$ and $\beta_{N+1,r}^2=(k_{N+1}^2-\alpha_r^2)^{1/2}$, where the square is root chosen such that $\sqrt{1}=1$ with branch cut along the negative imaginary axis. We assume that the wavenumbers $k_j$ and the quantities $\gamma_j$ in the subdomains $\Omega_j$ are positive real numbers.

Wood frequencies are those values of $k$ for which there exist indices $r_0$ such that $\alpha_{r_0}^2=k^2$. The well posedness of the equations~\eqref{system_t} was established in~\cite{Petit} in the case of two domains $\Omega_0$ and $\Omega_1$ separated by the periodic interface $\Gamma_0$ with $\gamma_0=\gamma_1=1$ and any real wavenumbers $k_0$ and~$k_1$, including Wood frequencies. The techniques presented in~\cite{Petit} are easily applicable to periodic configurations with arbitrary number of layers.  To the best of our knowledge, any attempt at establishing uniqueness of solutions of equations~\eqref{system_t} was based on the aforementioned techniques. However, certain  requirements~\cite{arens2010scattering} must be imposed on the material parameters $(k_j,\gamma_j), 0\leq j\leq N$ in order to establish rigorously the uniqueness of solutions of equations~\eqref{system_t} using those techniques. For the sake of completeness, we provide in Appendix~\ref{wp_proof} a proof of uniqueness of solutions of equations~\eqref{system_t} under the assumption of monotonicity of the wavenumber $k_j,0\leq j\leq N+1$ and $\gamma_j=1,0\leq j\leq N+1$. In general, for a fixed periodic layered configuration with material properties $\varepsilon_j$, the transmission problem~\eqref{system_t} has a unique solution with the exception of a discrete set of frequencies $\omega$ whose only accumulation point is infinity~\cite{arens2010scattering,bonnet1994guided,DobsonFriedman1992}. The same comprehensive reference~\cite{arens2010scattering} contains a proof of existence of solutions for the transmission problem~\eqref{system_t} using both variational and boundary-integral equation arguments.

\begin{figure}
\centering
\includegraphics[scale=1]{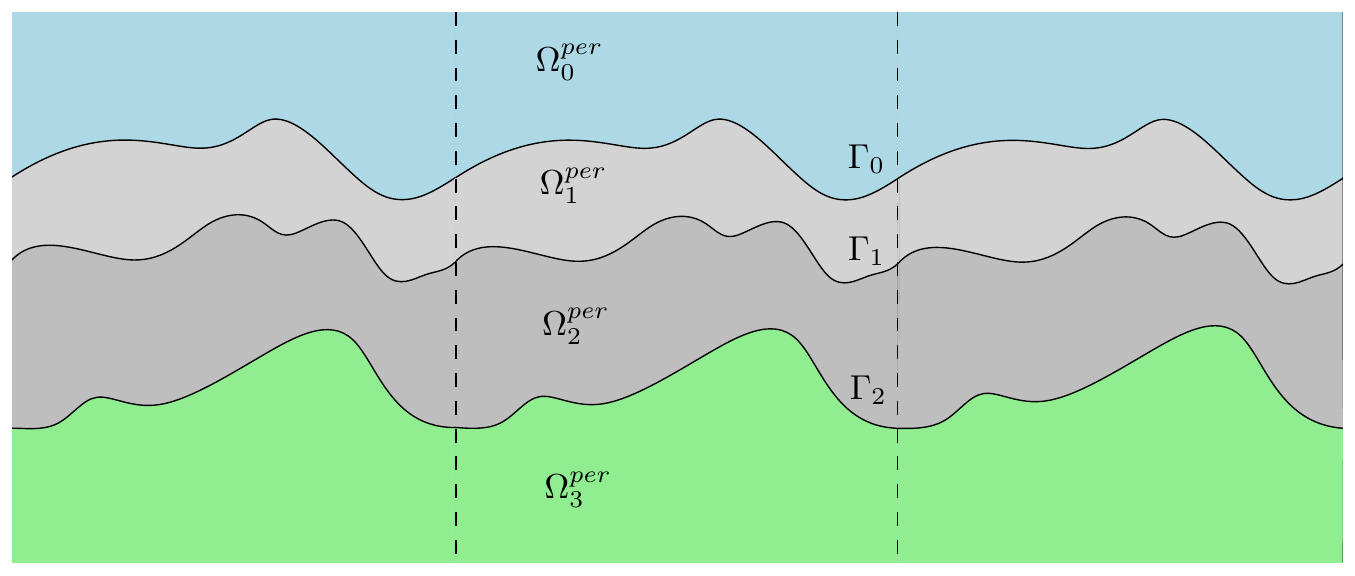}
\caption{Typical periodic layer structure with $N=2$; the $x_1$-axis is horizontal, and the $x_2$-axis is vertical.}
\label{fig:subdiv1}
\end{figure}

 \section{Domain decomposition approach\label{DDM}}

 We present a domain decomposition method (DDM) based on boundary-integral equations (BIEs) for the numerical solution of transmission problems~\eqref{system_t}. Just like BIE formulations, DDM formulations recast the original PDEs in terms of unknown quantities defined on the interfaces of material discontinuity. A non-overlapping domain decomposition approach for the solution of equations~\eqref{system_t} consists of solving Helmholtz subdomain problems in $\Omega_j,j=0,\ldots,N+1$ with matching Robin transmission boundary conditions on the common subdomain interfaces $\Gamma_j$ for $j=0,\ldots,N$. The main motivation for using DDM is the  seamless treatment of periodic configurations at Wood frequencies via BIE formulations, as well the ease with which it can handle inclusions in the periodic layers.  Specifically, DDM amount to computing $\alpha$-quasi-periodic  subdomain solutions:
\begin{eqnarray}\label{DDM_t}
  \Delta u_j +k_j^2 u_j &=&0\qquad {\rm in}\quad \Omega_j^{per},\\
  \gamma_0(\partial_{n_0}u_0+\partial_{n_0}u^\text{\tiny inc})-i\eta(u_0+u^\text{\tiny inc})&=&-\gamma_{1}\partial_{n_{1}}u_{1}-i\eta\ u_{1}\quad{\rm on}\quad \Gamma_{0}\nonumber\\
  \gamma_{1}\partial_{n_{1}}u_{1}-i\eta\ u_{1}&=&-\gamma_0(\partial_{n_0}u_0+\partial_{n_0}u^\text{\tiny inc})-i\eta(u_0+u^\text{\tiny inc})\quad{\rm on}\quad \Gamma_{0}\nonumber\\
  \gamma_j\partial_{n_j}u_j-i\eta\ u_j&=&-\gamma_{j+1}\partial_{n_{j+1}}u_{j+1}-i\eta\ u_{j+1}\quad{\rm on}\quad \Gamma_{j},\ 1\leq j\leq N\nonumber\\
  \gamma_{j+1}\partial_{n_{j+1}}u_{j+1}-i\eta\ u_{j+1}&=&-\gamma_j \partial_{n_j}u_j -i\eta\ u_j\quad{\rm on}\quad \Gamma_{j}, 1\leq j\leq N.\nonumber
\end{eqnarray}
In addition, we require that $u_0$ and $u_{N+1}$ be radiative and that $\eta>0$. The latter requirement ensures that the Robin problems in the semi-infinite domains $\Omega_0$ and $\Omega_{N+1}$ are well posed; see Theorem~\ref{wp_Omega_0}.

 The essence of the domain decomposition~\eqref{DDM_t}  is solving a Robin boundary-value problem in each layer subdomain and connecting the Robin boundary data across interfaces via the so-called Robin-to-Robin (RtR) maps~\cite{collino2000domain}---also see below. For a given layer subdomain $\Omega_j$ with $1\leq j\leq N$ we seek $w_j$ $\alpha$-quasi-periodic solutions of the following Helmholtz boundary-value problem
\begin{eqnarray}\label{eq:H}
  \Delta w_j+k_j^2w_j&=&0\quad{\rm in}\ \Omega_j^{per}\\
  \gamma_j\partial_{n_j}w_j-i\eta\ w_j&=&g_{j-1,j}\quad{\rm on}\ \Gamma_{j-1}\nonumber\\
  \gamma_j\partial_{n_j}w_j-i\eta\ w_j&=&g_{j,j}\quad{\rm on}\ \Gamma_{j}\nonumber
  \end{eqnarray}
where $g_{j-1,j}$ and $g_{j,j}$ are generic $\alpha$-quasi-periodic functions defined on $\Gamma_{j-1}$ and $\Gamma_j$. The RtR map $\mathcal{S}^j$ is defined as\begin{equation}\label{RtRboxj_t}
   \mathcal{S}^j\begin{bmatrix}g_{j-1,j}\\g_{j,j}\end{bmatrix}=\begin{bmatrix}(\gamma_j\partial_{n_j}w_j+i\eta\ w_j)|_{\Gamma_{j-1}}\\(\gamma_j\partial_{n_j}w_j+i\eta\ w_j)|_{\Gamma_{j}}\end{bmatrix}.
 \end{equation}
The computation of the RtR maps $\mathcal{S}^j$ requires solving the Helmholtz boundary value problem~\eqref{eq:H}. Of the two indices of the boundary data $g$ in equations~\eqref{eq:H}, the first index is the index of the interface and the second index is the index of the subdomain.  Thus, $g_{j-1,j}$ refers to boundary data on the interface $\Gamma_{j-1}$ on the side of the subdomain $\Omega_j$. The block structure of the RtR operators $\mathcal{S}^j$ defined in equation~\eqref{RtRboxj_t} is 
\begin{equation}\label{eq:block}
   \mathcal{S}^j\begin{bmatrix}g_{j-1,j}\\g_{j,j}\end{bmatrix}=\begin{bmatrix}\mathcal{S}^j_{j-1,j-1} & \mathcal{S}^j_{j-1,j}\\ \mathcal{S}^j_{j,j-1} & \mathcal{S}^j_{j,j}\end{bmatrix}\begin{bmatrix}g_{j-1,j}\\g_{j,j}\end{bmatrix}.
 \end{equation}
For the semi-infinite subdomain $\Omega_0$ $w_0$ is the $\alpha$-quasi-periodic outgoing solution of the Helmholtz boundary value problem
\begin{eqnarray}\label{eq:H0}
  \Delta w_0+k_0^2w_0&=&0\quad{\rm in}\ \Omega_0^{per}\\
  \gamma_0\partial_{n_0}w_0-i\eta\ w_0 &=&g_{0,0}\quad{\rm on}\ \Gamma_0\,,\nonumber
\end{eqnarray}
in which $g_{0,0}$ is a $\alpha$-quasi-periodic function defined on $\Gamma_0$, and we define the RtR map $\mathcal{S}^0$ by
\begin{equation}\label{eq:S0}
  \mathcal{S}^0g_{0,0}:=(\gamma_0\partial_{n_0}w_0+i\eta\ w_0)|_{\Gamma_0}.
\end{equation}
The RtR map $\mathcal{S}^{N+1}$ corresponding to the semi-infinite subdomain $\Omega_{N+1}$ is defined in a similar manner to $\mathcal{S}^0$  but  for boundary data $g_{N-1,N}$ defined on $\Gamma_N$. 

In DDM formulations~\eqref{DDM_t}, the unknown Robin data associated with each interface $\Gamma_j$
\begin{eqnarray*}
  f_{j}=\begin{bmatrix}f_{j,j}\\ f_{j,j+1}\end{bmatrix}&:=&\begin{bmatrix}(\gamma_j\partial_{n_j}u_j-i\eta\ u_j)|_{\Gamma_{j}}\\ (\gamma_{j+1}\partial_{n_{j+1}}u_{j+1}-i\eta\ u_{j+1})|_{\Gamma_j}\end{bmatrix},\ 0\leq j\leq N\\
\end{eqnarray*}
are matched via the subdomain RtR maps $\mathcal{S}^j,0\leq j\leq N+1$ giving rise to a $(2N+2)\times (2N+2)$ operator linear system.  The unknown Robin data $f=[f_0\ f_1\ \ldots f_{N}]^\top$ are the solution of the following linear system
\begin{equation}\label{ddm_t_exp}
  \mathcal{A}f={\rhs}\,,
\end{equation}
in which the DDM matrix $\mathcal{A}$ is a tridiagonal block matrix whose first two rows, and the rows indexed by $2j+1$ and $2j+2$ (corresponding to the unknown Robin data $f_{j,j}$ and $f_{j,j+1}$), and the last two rows, are given in explicit form
\begin{equation}\label{eq:mA}
  \mathcal{A}=\begin{bmatrix}
  I & \mathcal{S}^1_{0,0}& \mathcal{S}^1_{0,1} & \ldots & 0 & 0 & 0 & 0 & 0 &0 & \ldots & 0 & 0 \\
  \mathcal{S}^0 & I & 0 & \ldots & 0 & 0 & 0 & 0 & 0 & 0 & \ldots & 0 & 0\\
  \ldots & \ldots & \ldots & \ldots & \ldots & \ldots & \ldots & \ldots & \ldots & \ldots & \ldots & \ldots & \ldots\\
  0 & 0 & 0 & \ldots & 0 & 0 & I & \mathcal{S}^{j+1}_{j,j} & \mathcal{S}^{j+1}_{j,j+1} & 0 & \ldots  & 0 & 0\\
  0 & 0 & 0 & \ldots & 0 & \mathcal{S}^j_{j,j-1} & \mathcal{S}^j_{j,j} & I & 0 & 0 & \ldots & 0 & 0\\
  \ldots & \ldots & \ldots & \ldots & \ldots & \ldots & \ldots & \ldots & \ldots & \ldots & \ldots & \ldots & \ldots\\
  \ldots & \ldots & \ldots & \ldots & \ldots & \ldots & \ldots & \ldots & \ldots & \ldots & \ldots & I & \mathcal{S}^{N+1}\\
 \ldots & \ldots & \ldots & \ldots & \ldots & \ldots & \ldots & \ldots & \ldots & \ldots & \mathcal{S}^{N}_{N,N-1} & \mathcal{S}^{N}_{N,N} & I\\
  \end{bmatrix}\nonumber
  \end{equation}
 and in which the right-hand-side vector ${\rhs}=[\rhs_0\ \rhs_1\ \ldots\ \rhs_{N}]^\top$ has zero components $\rhs_\ell=[0\ 0]^\top,\ 1\leq \ell\leq N$, with the exception of the first component 
 \begin{equation}\label{rhs_ddm_t}
   \rhs_0=\begin{bmatrix}-(\gamma_0\partial_{n_0}u^\text{\tiny inc}-i\eta\ u^\text{\tiny inc})|_{\Gamma_0}\\-(\gamma_0\partial_{n_0}u^\text{\tiny inc}+i\eta\ u^\text{\tiny inc})|_{\Gamma_0}\end{bmatrix}\nonumber.
 \end{equation} 
 In what follows, we study spectral properties of the RtR operators $\mathcal{S}^j,0\leq j\leq N+1$, properties that will shed light onto the solvability of the DDM system~\eqref{ddm_t_exp}.

 \subsection{Spectral properties of the RtR operators~\label{S0}}

 The first question that arises is whether the RtR operators are properly defined under the assumptions on wavenumbers $k_j$ and coefficients $\gamma_j>0$, $j=0,\ldots,N+1$. We establish the following result, whose proof is essentially a simple extension of arguments presented in~\cite{Petit}.
 \begin{theorem}\label{wp_Omega_0}
   Let $w_0$ be the $\alpha$-quasi-periodic outgoing solution of the following Helmholtz equation 
\begin{eqnarray*}
  \Delta w_0+k_0^2w_0&=&0\quad{\rm in}\ \Omega_0^{per}\\
  \partial_{n_0}w_0-i\eta\ \gamma_0^{-1}w_0 &=&0\quad{\rm on}\ \Gamma_0.\nonumber
\end{eqnarray*}
Then $w_0$ is identically zero in $\Omega_0^{per}$.
 \end{theorem}
 \begin{remark} A similar uniqueness result holds for the homogeneous problem
 \begin{eqnarray*}
  \Delta w_{N+1}+k_{N+1}^2w_{N+1}&=&0\quad{\rm in}\ \Omega_{N+1}^{per}\\
  \partial_{n_{N+1}}w_{N+1}-i\eta\ \gamma_{N+1}^{-1}w_{N+1} &=&0\quad{\rm on}\ \Gamma_N.\nonumber
\end{eqnarray*}
 \end{remark}
 \begin{proof}
   Consider again $h>\max{F_0}$ and the domain $\Omega_{0,h}^{per}:=\{(x_1,x_2)\in\Omega_0^{per}: F_0(x_1)\leq x_2\leq h\}$. A simple application of Green's identities leads to 
\begin{eqnarray*}
  \int_{\Omega_{0,h}^{per}}(|\nabla w_0|^2-k_0^2|w_0|^2)dx&=&\int_{\Gamma_0}\partial_{n_0}w_0\ \overline{w_0}\ ds + \int_{\Gamma_{0,h}}\partial_{x_2}w_0\ \overline{w_0}\ dx_1\\
  &=&i\eta\gamma_0^{-1}\int_{\Gamma_0}|w_0|^2\ ds + \int_{\Gamma_{0,h}}\partial_{x_2}w_0\ \overline{w_0}\ dx_1
\end{eqnarray*}
where $\Gamma_{0,h}:=\{(x_1,x_2): 0\leq x_1\leq d,\ x_2=h\}$. Taking into account the fact that $w_0$ is radiating, we can express $w_0$ on the line segment $\Gamma_{0,h}$ in terms of the following Rayleigh series
\[
w_0(x_1,h)=\sum_{r\in\mathbb{Z}} C_r^{+}e^{i\alpha_r x_1+i\beta_{0,r} h}\,,
\]
from which it follows that
\[
\int_{\Gamma_{0,h}}\partial_{x_2}w_0\ \overline{w_0}\ dx_1=id\sum_{r\in\mathbb{Z},\ \beta_{0,r}>0}\beta_{0,r}|C_r^{+}|^2.
\]
Consequently,
\[
\int_{\Omega_{0,h}^{per}}(|\nabla w_0|^2-k_0^2|w_0|^2)dx=i\eta\gamma_0^{-1}\int_{\Gamma_0}|w_0|^2\ ds+id\sum_{r\in\mathbb{Z},\ \beta_{0,r}>0}\beta_{0,r}|C_r^{+}|^2.
\]
\notesps{The left-hand-side of this identity is real, whereas the right-hand side is a sum of non-negative imaginary terms, and thus each of these terms vanishes.}  This implies that $w_0=0$ on $\Gamma_0$, and thus $\partial_{n_0}w_0=0$ on $\Gamma_0$ as well. The result now follows from Holmgren's uniqueness theorem~\cite{Friedman2008}.
 \end{proof}

Consider now the following Helmholtz equation.   Let  $w_0$ be the $\alpha$-quasi-periodic outgoing solution of
\begin{eqnarray*}
  \Delta w_0+k_0^2w_0&=&0\quad{\rm in}\ \Omega_0^{per}\\
  \partial_{n_0}w_0-i\eta\ \gamma_0^{-1}w_0 &=&g_0\quad{\rm on}\ \Gamma_0
\end{eqnarray*}
where $g_0$ is a $\alpha$-quasi-periodic function defined on $\Gamma_0$. The matter of existence of such a solution will be settled in the next section through boundary-integral equation arguments. We are interested in estimating the norm of the RtR operator $\mathcal{S}^0$ as a continuous operator from $L^2_{per}(\Gamma_0)$ to itself.  We have
\[
\|g_0\|_2^2=\int_{\Gamma_0}(|\partial_{n_0}w_0|^2+\eta^2\gamma_0^{-2}|w_0|^2)ds-2\eta\gamma_0^{-1}\Im{\int_{\Gamma_0}\partial_{n_0}w_0\ \overline{w_0}\ ds}
\]
and
\[
\|\mathcal{S}^0g_0\|_2^2=\int_{\Gamma_0}(|\partial_{n_0}w_0|^2+\eta^2\gamma_0^{-2}|w_0|^2)ds+2\eta\gamma_0^{-1}\Im{\int_{\Gamma_0}\partial_{n_0}w_0\ \overline{w_0}\ ds}.
\]
Again, we have that
\[
\int_{\Omega_{0,h}^{per}}(|\nabla w_0|^2-k_0^2|w_0|^2)dx=\int_{\Gamma_0}\partial_{n_0}w_0\ \overline{w_0}\ ds + \int_{\Gamma_{0,h}}\partial_{x_2}w_0\ \overline{w_0}\ dx_1\,.
\]
Assuming the Rayleigh series expansion
\[
w_0(x_1,h)=\sum_{r\in\mathbb{Z}} C_r^{+}e^{i\alpha_r x_1+i\beta_{0,r} h}\,,
\]
we derive
\[
\Im{\int_{\Gamma_0}\partial_{n_0}w_0\ \overline{w_0}\ ds}=-d\sum_{r\in\mathbb{Z},\ \beta_{0,r}>0}\beta_{0,r}|C_r^{+}|^2\,,
\]
and hence
\[
\|\mathcal{S}^0g_0\|_2^2<\|g_0\|^2_2
\]
for all $g_0$ is a $\alpha$-quasi-periodic function defined on $\Gamma_0$.  It follows that $\|\mathcal{S}^0\|_{L^2_{per}(\Gamma_0)\to L^2_{per}(\Gamma_0)}\leq 1$. Similar arguments lead to the estimate $\|\mathcal{S}^{N+1}\|_{L^2_{per}(\Gamma_{N+1})\to L^2_{per}(\Gamma_{N+1})}\leq 1$. Green's identities establish the following theorem.

\begin{theorem}
  The RtR operators $\mathcal{S}^j$ are unitary in the space $L^2_{per}(\Gamma_j)\times L^2_{per}(\Gamma_{j+1})$ for all $j : 1\leq j\leq N$.
\end{theorem}
This unitarity can be used to establish the pointwise convergence of the Jacobi fixed-point iterations for the solution of the DDM formulation~\eqref{ddm_t_exp} by a relatively straightforward adaptation of the arguments presented in~\cite{collino2000domain} to the quasi-periodic setting.

\subsection{Calculations of RtR operators in terms of boundary-integral operators associated with quasi-periodic Green functions~\label{rtr}}
Implementation of DDM requires computation of RtR maps. We present in this section explicit representations of RtR maps in terms of boundary-integral operators associated with quasi-periodic Green functions that will serve as the basis of the implementation of the DDM algorithm. 
\subsubsection{Quasi-periodic Green functions, layer potentials and integral operators}
 For a given free-space wavenumber (normalized frequency) $k$, define the $\alpha$-quasi-periodic Green function
\begin{equation}\label{eq:qper_G} 
  G^q_k(x_,x_2)=\sum_{n\in\mathbb{Z}} e^{-i\alpha nd}G_k(x_1+nd,x_2)
\end{equation}
where $G_k(x_1,x_2)=\frac{i}{4}H_0^{(1)}(k|\mathbf{x}|),\ \mathbf{x}=(x_1,x_2)$.   Define $\alpha_r:=\alpha+\frac{2\pi}{d}r$ and $\beta_{r}=\beta_r(k):=(k^2-\alpha_r^2)^{1/2}$, where the branch of the square roots in the definition of $\beta_{r}$ is chosen in such a way that $\sqrt{1}=1$, and that the branch cut coincides with the negative imaginary axis. It can be shown that the series in the definition of the Green function $G^q_k$ in equation~\eqref{eq:qper_G} converge for wavenumbers $k$ for which none of the coefficients $\beta_r$ is equal to zero. In such cases it can be shown that $G^q_k$ can be expressed in the frequency domain in the form
\begin{equation}\label{eq:qper_G_freq}
  G^q_k(x_,x_2)=\frac{i}{2d}\sum_{r\in\mathbb{Z}}\frac{e^{i\alpha_r x_1+i\beta_r|x_2|}}{\beta_r}.
  \end{equation}

When the wavenumber $k$ is a Wood frequency, the set $W=W(k):=\{r_0\in\mathbb{Z}:\beta_{r_0}(k)=0\}$ is nonempty.
For wavenumbers that are Wood frequencies, the series in the definition of the Green function $G^q_k$ in equation~\eqref{eq:qper_G} does not converge. In the case when $k$ is a Wood frequency, we introduce the following shifted Green functions~\cite{Delourme}
\begin{equation}\label{eq:qper_G_shift}
  G^{q,j}_{k,h}(x_,x_2)=\sum_{n\in\mathbb{Z}} e^{-i\alpha nd}\sum_{\ell=0}^j(-1)^\ell \binom{j}{\ell}G_k(x_1+nd,x_2+\ell h)+\sum_{r\in W} c_r e^{i\alpha_r x_1+i(\sign{h})\beta_r x_2}
\end{equation}
for shifts $h\neq 0$, integers $j>0$, and non-zero coefficients $c_r\in\mathbb{C}$. The functions $G^{q,j}_{k,h}$ are radiating $\alpha$-quasi-periodic Green function in the halfplane $x_2>0$ for $h>0$ and respectively in the halfplane $x_2<0$ for $h<0$; these functions have poles at $x_1=0$ and $x_2=-\ell\ h, 0<\ell\leq j$. We note that the quantities $G^{q,j}_{k,h}$ defined in equation~\eqref{eq:qper_G_shift} still make sense when $k$ is not a Wood frequency, in which case the set $W$ can be defined as $W:=\{r_0\in\mathbb{Z}:|\beta_{r_0}|<\varepsilon\}$, where $\varepsilon$ is chosen to be sufficiently small; obviously, the set $W$ can be empty in some cases.

Assume now that the interface $\Gamma^{per}$ is defined as $\Gamma^{per}:=\{(x_1,F(x_1)):0\leq x_1\leq d\}$ where $F$ is a $C^2$ periodic function of principal period equal to $d$. Given a density $\varphi$ defined on $\Gamma^{per}$ (which can be extended by $\alpha$-quasi-periodicity to arguments $(x_1,F(x_1)), x_1\in\mathbb{R}$) we define the single-layer potentials corresponding to a wavenumber $k$
\begin{equation}\label{eq:single_layer}
  [SL_k^q\varphi](\mathbf{x}):=\int_{\Gamma^{per}}G^q_k(\mathbf{x},\mathbf{y})\varphi(\mathbf{y})ds(\mathbf{y})\qquad [SL_{k,h}^{q,j}\varphi](\mathbf{x}):=\int_{\Gamma^{per}}G^{q,j}_{k,h}(\mathbf{x},\mathbf{y})\varphi(\mathbf{y})ds(\mathbf{y})
\end{equation}
for $\mathbf{x}\notin\Gamma^{per}$ and $\mathbf{x}=(x_1,x_2)$ such that $0\leq x_1\leq d$.  The quantities $SL_k^q\varphi$ can be extended by $\alpha$-quasi-periodicity to define $\alpha$-quasi-periodic outgoing solutions of the Helmholtz equation corresponding to wavenumber $k$ in the domains $\{\mathbf{x}:x_2>F(x_1)\}$ and $\{\mathbf{x}:x_2<F(x_1)\}$. Similarly, the quantities $SL_{k,h}^{q,j}\varphi$ can be extended by $\alpha$-quasi-periodicity to define $\alpha$-quasi-periodic outgoing solutions of the Helmholtz equation corresponding to wavenumber $k$ in the domains $\{\mathbf{x}:x_2>F(x_1)\}$ for $h>0$ and respectively in the domain $\{\mathbf{x}:x_2<F(x_1)\}$ for $h<0$. Assuming that $\mathbf{n}$ is the unit normal to $\Gamma^{per}$ pointing into the domain $\{\mathbf{x}:x_2>F(x_1)\}$ one obtains the single layer potential on the interface $\Gamma^{per}$,
\begin{equation}\label{eq:SL}
  [S_k^q(\varphi)](\mathbf{x}) := \lim_{\varepsilon\to 0} [SL_k^q\varphi](\mathbf{x}\pm \varepsilon\mathbf{n}(\mathbf{x}))=\int_{\Gamma^{per}}G_k^q(\mathbf{x},\mathbf{y})\varphi(\mathbf{y})ds(\mathbf{y}),\quad \mathbf{x}\in\Gamma^{per}
\end{equation}
and
\begin{equation}\label{eq:SL_j_1}
  [S_{k,h}^{q,j}(\varphi)](\mathbf{x}):=\lim_{\varepsilon\to 0} [SL_{k,h}^{q,j}\varphi](\mathbf{x}+ \varepsilon\mathbf{n}(\mathbf{x}))=\int_{\Gamma^{per}}G_{k,h}^{q,j}(\mathbf{x},\mathbf{y})\varphi(\mathbf{y})ds(\mathbf{y}),\quad \mathbf{x}\in\Gamma^{per},\qquad h>0
\end{equation}
as well as
\begin{equation}\label{eq:SL_j_2}
  [S_{k,h}^{q,j}(\varphi)](\mathbf{x}):=\lim_{\varepsilon\to 0} [SL_{k,h}^{q,j}\varphi](\mathbf{x}- \varepsilon\mathbf{n}(\mathbf{x}))=\int_{\Gamma^{per}}G_{k,h}^{q,j}(\mathbf{x},\mathbf{y})\varphi(\mathbf{y})ds(\mathbf{y}),\quad \mathbf{x}\in\Gamma^{per},\qquad h<0.
\end{equation}
Also, we have
\begin{equation}\label{eq:DLT}
  \lim_{\varepsilon\to 0} \nabla[SL_k^q\varphi](\mathbf{x}\pm \varepsilon\mathbf{n}(\mathbf{x}))\cdot\mathbf{n}(\mathbf{x})=\mp \frac{1}{2}\varphi(\mathbf{x})+[(K_k^q)^\top(\varphi)](\mathbf{x}),\quad \mathbf{x}\in\Gamma^{per}
\end{equation}
and
\begin{equation}\label{eq:DLT_j_1}
  \lim_{\varepsilon\to 0} \nabla[SL_{k,h}^{q,j}\varphi](\mathbf{x}+ \varepsilon\mathbf{n}(\mathbf{x}))\cdot\mathbf{n}(\mathbf{x})=-\frac{1}{2}\varphi(\mathbf{x})+[(K_{k,h}^{q,j})^\top(\varphi)](\mathbf{x}),\quad \mathbf{x}\in\Gamma^{per},\qquad h>0
\end{equation}
as well as
\begin{equation}\label{eq:DLT_j_2}
  \lim_{\varepsilon\to 0} \nabla[SL_{k,h}^{q,j}\varphi](\mathbf{x}- \varepsilon\mathbf{n}(\mathbf{x}))\cdot\mathbf{n}(\mathbf{x})=\frac{1}{2}\varphi(\mathbf{x})+[(K_{k,h}^{q,j})^\top(\varphi)](\mathbf{x}),\quad \mathbf{x}\in\Gamma^{per},\qquad h<0.
\end{equation}
In equations~\eqref{eq:DLT}, the adjoint double-layer operators can be defined explicitly as
\begin{equation}\label{eq:DLT_explicit}
  [(K_k^q)^\top(\varphi)](\mathbf{x})=\int_{\Gamma^{per}}\frac{\partial G_k^q(\mathbf{x},\mathbf{y})}{\partial\mathbf{n}(\mathbf{x})}\varphi(\mathbf{y})ds(y),\quad \mathbf{x}\in\Gamma^{per}
\end{equation}
with similar definitions for the operators defined in equations~\eqref{eq:DLT_j_1} and~\eqref{eq:DLT_j_2} respectively.
\subsubsection{Boundary-integral  representation of RtR maps}
Having defined the $\alpha$-quasi-periodic boundary-integral operators above, we are now in a position to compute the various RtR operators $\mathcal{S}^j$. We start with the RtR operator $\mathcal{S}^0$ corresponding to problem~\eqref{eq:H0}. We define $Z_0=i\eta\gamma_0^{-1}$ and seek $w_0$ in the form
\[
w_0:=SL^q_{k_0}\varphi_0,
\]
in which the single-layer potential $SL^q_{k_0}$ is defined in~\eqref{eq:single_layer} integrating on the curve $\Gamma_0$.  From the equation
\begin{equation}
  \mathcal{S}^0 [(1/2) I + K - Z_0 S_{\Gamma_0,k_0}^q]
   = [(1/2) I + K - Z_0 S_{\Gamma_0,k_0}^q]\,,
\end{equation}
we obtain an explicit formula for the RtR operator $\mathcal{S}^0$ defined in equation~\eqref{eq:S0},
\begin{equation}\label{eq:calc_S0}
  \mathcal{S}^0=I+2Z_0S_{\Gamma_0,k_0}^q\left(\frac{1}{2}I+(K_{\Gamma_0,k_0}^q)^\top-Z_0S_{\Gamma_0,k_0}^q\right)^{-1},
\end{equation}
in which the operators $(K_{\Gamma_0,k_0}^q)^\top$ are defined just as in equations~\eqref{eq:DLT_explicit} but with unit normal $n_0$ pointing into $\Omega^{-}_0$ (exterior of $\Omega_0$). Here and in what follows we introduce an additional subscript to make explicit the curve that is the domain of integration of the boundary-integral operators. 

The invertibility of the operator featured in equations~\eqref{eq:calc_S0} can be established in a straightforward manner.
\begin{theorem}\label{well_p_0}
  Under the assumptions that $F_0$ is $C^2$, and $k_0$ is not a Wood frequency, the operator
  \[
  \mathcal{A}_0:=\frac{1}{2}I+(K_{\Gamma_0,k_0}^q)^\top-Z_0S_{\Gamma_0,k_0}^q,\quad \mathcal{A}_0:L^2_{per}(\Gamma_0)\to L^2_{per}(\Gamma_0)
  \]
  is invertible with continuous inverse. 
\end{theorem}
\begin{proof}
Because of the regularity of the boundary $\Gamma_0$, both operators $(K_{\Gamma_0,k_0}^q)^\top:L^2_{per}(\Gamma_0)\to L^2_{per}(\Gamma_0)$ and $S_{\Gamma_0,k_0}^q:L^2_{per}(\Gamma_0)\to L^2_{per}(\Gamma_0)$ are compact. Thus, the conclusion of the Theorem follows once we establish the injectivity of the operator $\mathcal{A}_0$. Let $\varphi_0\in Ker(\mathcal{A}_0)$ and define $v_0:=SL_{k_0}^q\varphi_0$ in $\mathbb{R}^2\setminus \{(x_1,F_0(x_1)), x_1\in\mathbb{R})\}$.  The function $v_0$ is a radiating $\alpha$-quasi-periodic solution of the Helmholtz equation in $\Omega_0$ with impedance boundary conditions $\partial_{n_0}v_0-Z_0v_0=0$ on $\Gamma_0$, and thus, by Theorem~\ref{wp_Omega_0}, we have that $v_0$ is identically zero in $\Omega_0$. In particular, it follows that $v_0=0$ on $\Gamma_0$. Thus, $v_0$ is a radiating, $\alpha$-quasi-periodic solution of the Helmholtz equation with wavenumber $k_0$ in the domain $\Omega_0^{-}:=\{(x_1,x_2):x_2<F_0(x_1)\}$ with zero Dirichlet boundary conditions on $\Gamma_0$. This implies that $v_0$ is identically zero in the domain $\Omega_0^{-}$ by uniqueness the Dirichlet problem. Finally, the jump conditions of the normal derivatives of single-layer potentials imply that $\varphi_0=0$ on $\Gamma_0$.
\end{proof}

Alternatively, we can seek $w_0$ in the form
\[
w_0:=SL^{q,j}_{k_0,h}\varphi\,,
\]
from which we obtain a representation of the RtR operators $\mathcal{S}^0$ in the form
\begin{equation}\label{eq:calc_S0_h}
  \mathcal{S}^0=I+2Z_0S_{\Gamma_0,k_0,h}^{q,j}\left(\frac{1}{2}I+(K_{\Gamma_0,k_0,h}^{q,j})^\top-Z_0S_{\Gamma_0,k_0,h}^{q,j}\right)^{-1},\quad h>0.
\end{equation}
The invertibility of the operators that feature in equation~\eqref{eq:calc_S0_h} is much more subtle.  It can be established by modification of arguments presented in a recent paper of some of the authors~\cite{bruno2017three}.  The proof is presented there for doubly periodic layered media in three dimensions.  The theorem below is also valid in three dimensions; its proof in the two-dimensional reduction is given in Appendix~\ref{proof_Wood}. 

\begin{theorem}\label{well posedness_Wood}
  Under the assumption that $F_0$ is $C^2$ and that $k_0$ is a Wood frequency, the operator
  \[
  \mathcal{A}_{0,h}:=\frac{1}{2}I+(K_{\Gamma_0,k_0,h}^{q,j})^\top-Z_0S_{\Gamma_0,k_0,h}^{q,j},\ j\geq 1, \mathcal{A}_{0,h}:L^2_{per}(\Gamma_0)\to L^2_{per}(\Gamma_0)
  \]
  is invertible with continuous inverse for all but a discrete set of values of the shift $h>0$.
\end{theorem}

Note that the calculations of the RtR maps $\mathcal{S}^{N+1}$ can be performed similarly to arrive at
\begin{equation}\label{eq:calc_SN}
  \mathcal{S}^{N+1}=I+2Z_{N+1}S_{\Gamma_{N},k_{N+}}^q\left(\frac{1}{2}I+(K_{\Gamma_{N},k_{N+1}}^q)^\top-Z_{N+1}S_{\Gamma_{N},k_{N+1}}^q\right)^{-1},
\end{equation}
\begin{equation}\label{eq:calc_SN_h}
  \mathcal{S}^{N+1}=I+2Z_{N+1}S_{\Gamma_{N},k_{N+1},h}^{q,j}\left(\frac{1}{2}I+(K_{\Gamma_{N},k_{N+1},h}^{q,j})^\top-Z_{N+1}S_{\Gamma_{N},k_{N+1},h}^{q,j}\right)^{-1},\quad h<0.
\end{equation}
where $Z_{N+1}=i\eta\gamma_{N+1}^{-1}$ and the adjoint double-layer operators are defined with respect to the unit normal $n_{N+1}$ pointing outside of the domain $\Omega_{N+1}$. The invertibility of the operators in equations~\eqref{eq:calc_SN} and~\eqref{eq:calc_SN_h} can be established analogously to the results in Theorems~\ref{well_p_0} and~\ref{well posedness_Wood}.

Finally, the RtR maps for the domains $\Omega_j,1\leq j<N$ can be expressed in closed form via boundary-integral equations. Indeed, consider a generic domain $\Omega^{per}:=\{(x_1,x_2):0\leq x_1\leq d, F_b(x_1)\leq x_2\leq F_t(x_1)\}$ where $F_t$ and $F_b$ are $d$-periodic $C^2$ functions. Let us denote by $\Gamma_t:=\{(x_1,x_2):0\leq x_1\leq d, x_2=F_t(x_1)\}$, $\Gamma_b:=\{(x_1,x_2):0\leq x_1\leq d, x_2=F_b(x_1)\}$, and let $n$ denote the unit normal to $\Gamma_t\cup\Gamma_b$ pointing outside of the domain $\Omega^{per}$---see Figure~\ref{fig:layer}. Then the Helmholtz problems~\eqref{eq:H} can be all expressed in the generic form
\begin{eqnarray}\label{eq:Hgen}
  \Delta w+k^2w&=&0\quad{\rm in}\ \Omega^{per}\\
  \partial_{n}w-Z\ w&=&g_{t}\quad{\rm on}\ \Gamma_{t}\nonumber\\
  \partial_{n}w-Z\ w&=&g_{b}\quad{\rm on}\ \Gamma_{b}\nonumber
  \end{eqnarray}
where $g_t$ and $g_b$ are $\alpha$-quasi-periodic functions and $\Im{Z}>0$.  The RtR operators $\mathcal{S}^j,1\leq j<N$ are related to the following RtR operator associated with the Helmholtz problems~\eqref{eq:Hgen}:
\begin{equation}\label{RtRboxj_O}
   \mathcal{S}\begin{bmatrix}g_{t}\\g_{b}\end{bmatrix}:=\begin{bmatrix}(\partial_{n}w+Z\ w)|_{\Gamma_{t}}\\(\partial_{n}w+Z\ w)|_{\Gamma_{b}}\end{bmatrix}.
 \end{equation}
Seeking the solution $w$ of equations~\eqref{eq:Hgen} in the form
\[
w = SL_{k,t}^q\varphi_t+SL_{k,b}^q\varphi_b\,,
\]
in which $SL_{k,t}^q$ ($SL_{k,b}^q$) denotes the quasi-periodic single-layer potential whose domain of integration in $\Gamma_t$ ($\Gamma_b$), we arrive at the following expression for the RtR operator $\mathcal{S}$: 
\begin{equation}\label{eq:S_layer_j}
  \mathcal{S}=\begin{bmatrix}I &0\\0& I\end{bmatrix}+2Z\begin{bmatrix}S^q_{k,t,t} & S^q_{k,b,t}\\ S^q_{k,t,b} & S^q_{k,b,b}\end{bmatrix}\begin{bmatrix}(1/2)I+(K^q_{k,t,t})^\top +Z S^q_{k,t,t} & (K^q_{k,b,t})^\top+Z S^q_{k,b,t}\\ (K^q_{k,t,b})^\top+Z S^q_{k,t,b} & (1/2)I+(K^q_{k,b,b})^\top +Z S^q_{k,b,b}\end{bmatrix}^{-1}.
\end{equation}
We note that in equation~\eqref{eq:S_layer_j}, the subscripts in the notation $S^q_{k,b,t}$ signify that in equation~\eqref{eq:SL} the target point $\mathbf{x}\in\Gamma_t$ and the integration point $\mathbf{y}\in\Gamma_t$, whereas in the notation $S^q_{k,b,t}$ signify that in equation~\eqref{eq:SL} the target point $\mathbf{x}\in\Gamma_t$ and the integration point $\mathbf{y}\in\Gamma_b$.  All the other additional subscripts in equation~\eqref{eq:S_layer_j} have similar meanings related to the locations of target and integration points for single and adjoint double-layer boundary-integral operators. The invertibility of the operators featuring in equation~\eqref{eq:S_layer_j} can be established using similar reasoning to that in the proof of Theorem~\ref{well_p_0}.

In the case when $k$ is a Wood frequency, an equivalent representation of the RtR operator $\mathcal{S}$ can be obtained if we replace the quasi-periodic boundary-integral operators in equation~\eqref{eq:S_layer_j} by shifted quasi-periodic boundary-integral operators, provided the shift $h>0$ is chosen larger than the width of the domain $\Omega^{per}$; the latter requirement is needed to ensure that no poles of the shifted quasi-periodic functions are contained in the domain $\Omega^{per}$. The invertibility of the ensuing matrix operator is the subject of the following theorem:
\begin{theorem}\label{inv_layer_Wood}
  Assume $k$ is a Wood frequency. Then the operator
  \[
  \mathcal{A}_h:=\begin{bmatrix}(1/2)I+(K^{q,j}_{k,h,t,t})^\top +Z S^{q,j}_{k,h,t,t} & (K^{q,j}_{k,h,b,t})^\top+Z S^{q,j}_{k,h,b,t}\\ (K^{q,j}_{k,h,t,b})^\top+Z S^{q,j}_{k,h,t,b} & (1/2)I+(K^{q,j}_{k,h,b,b})^\top +Z S^{q,j}_{k,h,b,b}\end{bmatrix}
  \]
  is invertible with continuous inverse in the space $L^2_{per}(\Gamma_t)\times L^2_{per}(\Gamma_b)$ for all but a discrete set of values of the shift $h>0$.
\end{theorem}
\begin{proof} Given that all the boundary-integral operators that enter the definition of the matrix operator $\mathcal{A}_h$ are compact in $L^2_{per}(\Gamma_t)\times L^2_{per}(\Gamma_b)$, the result follows once we establish the injectivity of the operator $\mathcal{A}_h$. Let $(\varphi_t,\varphi_b)\in Ker(\mathcal{A}_h)$ and define
  \[
  w = SL_{k,h,t}^{q,j}\varphi_t+SL_{k,h,b}^{q,j}\varphi_b\quad{\rm in}\ \mathbb{R}^2\setminus(\Gamma_t\cup\Gamma_b).
  \]
  Clearly $w$ is a $\alpha$-quasi-periodic solution of equation~\eqref{eq:Hgen} with zero Robin boundary conditions on $\Gamma_t$ and $\Gamma_b$, and as such $w=0$ in $\Omega^{per}$. In particular, $w$ vanishes on $\Gamma_t$. Also, given that the shift $h$ is chosen so that the poles of $G^{q,j}_{k,h}$ are in the domain $\Omega_b^{-}:=\{(x_1,x_2):0\leq x_1\leq d, x_2\leq F_b(x_1)\}$, $w$~is a radiating $\alpha$-quasi-periodic solution of the Helmholtz equation in the domain $\Omega_t^{+}=\{(x_1,x_2):0\leq x_1\leq d, F_t(x_1)\leq x_2\}$, which vanishes on $\Gamma_t$. This means that $w=0$ in $\Omega^{+}$~\cite{Petit}. Using the jump conditions of the normal derivatives of single-layer potentials across $\Gamma_t$, we get that $\varphi_t=0$ on $\Gamma_t$. Accordingly, we have that
  \[
  w = SL_{k,h,b}^{q,j}\varphi_b\quad{\rm in}\ \mathbb{R}^2\setminus \Gamma_b
  \]
  vanishes in the domain $\Omega_b^{+}=\{(x_1,x_2):0\leq x_1\leq d, F_b(x_1)\leq x_2\}$. The arguments in the proof of Theorem~\ref{well posedness_Wood} can be repeated verbatim to conclude that $\varphi_b=0$ on $\Gamma_b$.
  \end{proof}

\begin{remark}\label{inclusions}
  The computation of the layer RtR maps described in equations~\eqref{eq:S_layer_j} can be extended in a straightforward manner to the case when impenetrable or penetrable inclusions are present in the domain $\Omega^{per}$.  In this case, the matrix $\mathcal{S}$ in equations~\eqref{eq:S_layer_j} needs be augmented by blocks that account for the interactions of the inclusions $D$ with $\Gamma_t$ and $\Gamma_b$, as well as its self-interactions that account for the boundary conditions to be imposed on $\partial D$. 
  \end{remark}

\begin{figure}
\centering
\includegraphics[scale=1]{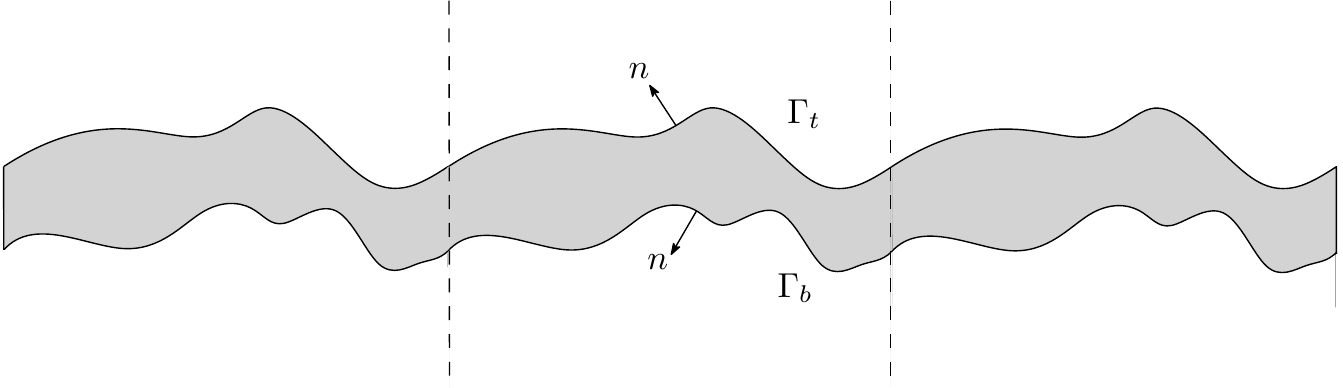}
\caption{Typical middle-layer structure.}
\label{fig:layer}
\end{figure}

\section{DDM Nystr\"om discretization\label{nystrom}}

Our numerical solution of equations~\eqref{ddm_t_exp}
relies on Nystr\"om discretizations of the boundary-integral operators featured in the computation of the RtR operators given in Section~\ref{rtr}. In order to speed up the notoriously slow convergence of the quasi-periodic Green function $G^q_k$ defined in equation~\eqref{eq:qper_G} for frequencies that are away from Wood frequencies, we make use of the recently introduced windowed Green function Method~\cite{bruno2016superalgebraically,bruno2017three,Delourme}. Specifically, let $\chi(r)$ be a smooth cutoff function equal to $1$ for $r<r_1$ and equal to $0$ for $r>r_2$ ($0<r_1<r_2$) and define the windowed Green functions
\begin{equation}\label{eq:qper_GA}
  G^{q,A}_k(x_,x_2)=\sum_{n\in\mathbb{Z}} e^{-i\alpha nd}G_k(x_1+nd,x_2)\chi(r_n/A),\quad r_n=((x_1+nd)^2+x_2^2)^{1/2}
\end{equation}
and
\begin{eqnarray}\label{eq:qper_G_shiftA}
  G^{q,j,A}_{k,h}(x_,x_2)&=&\sum_{n\in\mathbb{Z}} e^{-i\alpha nd}\sum_{\ell=0}^j(-1)^\ell \binom{j}{\ell}G_k(x_1+nd,x_2+\ell h)\chi(r_{n,\ell}/A)\nonumber\\
  &+&\sum_{r\in W} c_r e^{i\alpha_r x_1+i(\sign{h})\beta_r x_2}, \quad r_{n,\ell}=((x_1+nd)^2+(x_2+\ell h)^2)^{1/2}.
\end{eqnarray}
On account of the windowed function $\chi$, the summations in equations~\eqref{eq:qper_GA} and~\eqref{eq:qper_G_shiftA} are over a finite range of indices $n$. The functions $G^{q,A}_k$ were shown to converge superalgebraically fast to $G^q_k$ as $A\to\infty$ when $k$ is not a Wood frequency~\cite{bruno2016superalgebraically,bruno2017three,Delourme}, whereas the functions $G^{q,j,A}_{k,h}$ were shown to converge algebraically fast to a $\alpha$-quasi-periodic Green function as $A\to\infty$ (the rate increases as the number of shifts $j$ grows) in the half-plane $x_2>0$ when $h>0$ and respectively $x_2<0$ when $h<0$ for all frequencies $k$, including at and around Wood frequencies~\cite{Delourme}.

Our Nystr\"om discretizations rely on trigonometric collocation in two dimensions. As such, we reformulate the DDM system in terms of {\em periodic} quantities by extracting the phase $e^{-i\alpha x_1}$ from all Robin data, the right-hand side, as well as RtR maps. The calculation of the RtR maps is performed via boundary-integral operators acting on periodic densities $\widetilde{\varphi}$ defined as $\widetilde{\varphi}(x_1,x_2):=\varphi(x_1,x_2)e^{-i\alpha x_1}$ and periodic kernels $e^{i\alpha(x_1-y_1)}G_k^q(x_1,x_2;y_1,y_2)$. Furthermore, the discretization of the boundary-integral operators featured in Section~\ref{rtr} is done by replacing the Green functions $G^q_k$ and $G^{q,j}_k$ in their definitions by the fast convergent windowed approximations $G^{q,A}_k$ and $G^{q,j,A}_{k,h}$ defined in equations~\eqref{eq:qper_GA} and~\eqref{eq:qper_G_shiftA} respectively.  Finally, our numerical scheme requires a simple modification of the Martensen-Kussmaul (MK) periodic logarithmic splitting Nystr\"om approach~\cite{kusmaul,martensen} in order to enable high-order evaluations of boundary-integral operators whose kernels are windowed periodic Green functions---full details of this approach are given in~\cite{Delourme}. In a nutshell, boundary-integral operators that feature the windowed Green functions $G^{q,A}_k$ and $G^{q,j,A}_{k,h}$ defined in equations~\eqref{eq:qper_GA} and~\eqref{eq:qper_G_shiftA} respectively are recast in a form that involves integration around target points $\mathbf{x}$ but with domains of integration that span the whole real axis---the latter is achieved via the windowing functions $\chi$ and by periodic extensions of the densities $\widetilde{\varphi}$. This setting allows for a direct extension of the periodic logarithmic splitting of the Green functions that is central to MK Nystr\"om approach. In three dimensions, our Nystr\"om discretizations also rely on {\em global} trigonometric interpolation, use of floating partitions of unity and analytic resolution of singularities, as well as the use of windowed Green functions~\cite{bruno2017three}. Interestingly, using global trigonometric interpolation in conjunction with changes of variables to polar coordinates in the resolution of Green function singularities allows for straightforward constructions of Nystr\"om collocation matrices of three-dimensional boundary-integral operators  via two-dimensional Discrete Fourier Transform matrices. The availability of such Nystr\"om collocation matrices in three dimensions plays an important role in the efficient computations of RtR maps, as we explain next.

Following the prescriptions outlined above, a boundary-integral operator whose kernel is a windowed Green function (or its normal derivative) acting on a periodic density $\widetilde{\varphi}$ and whose domain of integration is a generic curve $\Gamma^{per}$ per the definition given in Section~\ref{rtr} (i.e. $\Gamma^{per}:=\{(x_1,F(x_1)):0\leq x_1\leq d\}$, where $F$ is a $C^2$ periodic function of principal period $d$) is Nystr\"om discretized as a $M\times M$ matrix where the periodic density $\widetilde{\varphi}$ is trigonometrically collocated at the equi-spaced mesh $\{(t_\ell,F(t_\ell)): t_\ell=\ell d/M, 0\leq \ell<M=2m\}$. For a fixed $M=2m, m>0$, assuming that the Robin data $f_j=[f_{j,j}\ f_{j,j+1}]^\top$ on each interface $\Gamma_j,0\leq j\leq N$ is collocated at the mesh $L_j:=\{(t_\ell,F_j(t_\ell)): t_\ell=\ell d/M, 0\leq \ell<M\}$, it follows that the RtR maps $\mathcal{S}^0$ and $\mathcal{S}^{N+1}$ are discretized as $M\times M$ Nystr\"om matrices $\mathcal{S}^0_M$ and $\mathcal{S}^{N+1}_M$ via Nystr\"om discretizations of the boundary-integral operators featured in equations~\eqref{eq:calc_S0} and~\eqref{eq:calc_SN} respectively in the case when neither $k_0$ nor $k_{N+1}$ are Wood frequencies or in equations~\eqref{eq:calc_S0_h} and~\eqref{eq:calc_SN_h} respectively in the case when $k_0$ and $k_{N+1}$ are Wood frequencies (same considerations apply in three dimensions). We note that according to equations~\eqref{eq:calc_S0} and~\eqref{eq:calc_SN} (and their analogues~\eqref{eq:calc_S0_h} and~\eqref{eq:calc_SN_h}), the calculation of the Nystr\"om matrices $\mathcal{S}^0_M$ and $\mathcal{S}^{N+1}_M$ require inversions of $M\times M$ matrices, which is done using LU factorizations. Similarly, the RtR maps $\mathcal{S}^j,1\leq j\leq N$ are discretized as $(2M)\times (2M)$ Nystr\"om matrices $\mathcal{S}^j_M$  via Nystr\"om discretizations of the boundary-integral operators featured in equations~\eqref{eq:S_layer_j}, and their calculations require, in turn, inversions of $(2M)\times(2M)$ matrices; these inversions are also performed through LU factorizations. It is also possible to employ Schur complements to perform the inversion of the matrices needed in the calculations of the RtR maps $\mathcal{S}^j,1\leq j\leq N$---see the proof of Theorem~\ref{Fredholm}; in that case matrices of size $M\times M$ need be inverted. The Nystr\"om discretization matrices $\mathcal{S}^j_M$ are further expressed in $M\times M$ block form
\begin{equation*}
   \mathcal{S}^j_M=\begin{bmatrix}\mathcal{S}^j_{j-1,j-1,M} & \mathcal{S}^j_{j-1,j,M}\\ \mathcal{S}^j_{j,j-1,M} & \mathcal{S}^j_{j,j,M}\end{bmatrix},
\end{equation*}
where each of the matrices above constitutes a Nystr\"om discretization matrix of the operators on the right-hand-side of equation~\eqref{eq:block}.  

Based on these Nystr\"om discretizations of RtR maps, the DDM algorithm computes matrix approximations of all the RtR maps needed. Clearly, the procedure outlined above allows us to assemble a block tridiagonal $2M(N+1)\times 2M(N+1)$ Nystr\"om discretization matrix $A_M$ of the operator matrix $A$ in equation~\eqref{eq:mA}, where each operator block in equation~\eqref{eq:mA} is replaced by its corresponding Nystr\"om discretization matrix. The computation of collocated Robin data $f_{j,M}$ at the grids $L_j$ for $0\leq j\leq N$ requires solution of a linear system featuring the matrix $A_M$. We present in what follows an efficient algorithm for the solution of this system that eliminates sequentially  the unknowns $f_{j,M}, 0\leq j\leq N$ using $N+1$ stages of recursive Schur complements; storage of all of the non-zero blocks in the matrix $A_M$ is not required by this algorithm. The key technical ingredient is that in the case when matrices
\[
\mathcal{D}:=\begin{bmatrix} I & \mathcal{A} \\ \mathcal{B} & I \end{bmatrix}
\]
are invertible, then their inverses can be computed explicitly
\begin{equation}\label{inv_matrix_explicit}
  \mathcal{D}^{-1}=\begin{bmatrix}
I+\mathcal{A}(I-\mathcal{B}\mathcal{A})^{-1}\mathcal{B} & -\mathcal{A}(I-\mathcal{B}\mathcal{A})^{-1}\\
-(I-\mathcal{B}\mathcal{A})^{-1}\mathcal{B} &
(I-\mathcal{B}\mathcal{A})^{-1}
\end{bmatrix}.
\end{equation}

The Schur complement elimination algorithm begins with

{\em Stage 1: elimination of the unknowns $f_{0,M}$}. We express the discrete DDM system in the following block form that separates the contribution of the Robin data $f_{0,M}$ from the rest of the Robin data. In detail,
\begin{eqnarray*}
  \begin{bmatrix}\mathcal{D}_{0,M}& \mathcal{A}_{0,M}\\\mathcal{B}_{0,M} & \mathcal{C}_{0,M} \end{bmatrix}\begin{bmatrix}f_{0,M}\\\widetilde{f}_{0,M}\end{bmatrix}&=&\begin{bmatrix}\rhs_{0,M}\\0_{2NM,1}\end{bmatrix}\\
      \mathcal{D}_{0,M}&=&\begin{bmatrix}I_M&\mathcal{S}^{1}_{0,0,M}\\\mathcal{S}^0_M & I_M\end{bmatrix}\\
        \mathcal{A}_{0,M}&=&\begin{bmatrix}\mathcal{S}^{1}_{0,1,M}& 0_M & 0_{M,2(N-1)M}\\ 0_M & 0_M & 0_{M,2(N-1)M}\end{bmatrix}\\
        \mathcal{B}_{0,M}&=&\begin{bmatrix}0_M & 0_M\\ 0_M & \mathcal{S}^1_{1,0,M}\\ 0_{2(N-1)M,M} & 0_{2(N-1)M,M}\end{bmatrix}\\
        \mathcal{C}_{0,M}&=&\begin{bmatrix} I_M & \mathcal{S}^2_{1,1,M}& \cdots \\ \mathcal{S}^1_{1,1,M} & I_M & \cdots\\ \cdots & \cdots &\cdots \end{bmatrix},
\end{eqnarray*}
where $\widetilde{f}_{0,M}=[f_{1,M}\ f_{2,M}\ \ldots f_{N,M}]^\top$. In the notations above and in what follows, we make explicit the matrix size of various zero matrices; for instance, the notation $0_{p,q}$ denotes a zero matrix with $p$ rows and $q$ columns, and $0_p$ denotes the zero $p\times p$ matrix.  We have
\[
f_{0,M}=-\mathcal{D}_{0,M}^{-1}\mathcal{A}_{0,M}\widetilde{f}_{0,M}+\mathcal{D}_{0,M}^{-1}\rhs_{0,M},
\]
and hence
\[
(\mathcal{C}_{0,M}-\mathcal{B}_{0,M}\mathcal{D}_{0,M}^{-1}\mathcal{A}_{0,M})\widetilde{f}_{0,M}=\rhs_{1,M},\ \rhs_{1,M}:=-\mathcal{B}_{0,M}\mathcal{D}_{0,M}^{-1}\rhs_{0,M},
\]
which can be written in expanded form using formula~\eqref{inv_matrix_explicit} to compute the inverse of $\mathcal{D}_{0,M}$:
\begin{eqnarray*}
  \begin{bmatrix}\mathcal{D}_{1,M}& \mathcal{A}_{1,M}\\\mathcal{B}_{1,M} & \mathcal{C}_{1,M} \end{bmatrix}\begin{bmatrix}f_{1,M}\\\widetilde{f}_{1,M}\end{bmatrix}&=&\begin{bmatrix}\rhs_{1,M}\\0_{2(N-1)M,1}\end{bmatrix}\\
    \rhs_{1,M}&=&-\mathcal{B}_{0,M}\mathcal{D}_{0,M}^{-1}\rhs_{0,M}\\
    \mathcal{D}_{1,M}&=&\begin{bmatrix}I_M&\mathcal{S}^{2}_{1,1,M}\\\mathcal{S}^{top}_{1,M} & I_M\end{bmatrix}\\
      \mathcal{S}^{top}_{1,M}&=&\mathcal{S}^1_{1,0,M}(I_M-\mathcal{S}^0_M\mathcal{S}^1_{0,0,M})^{-1}\mathcal{S}^0_{M}\mathcal{S}^1_{0,1,M}+\mathcal{S}^1_{1,1,M}\\
        \mathcal{A}_{1,M}&=&\begin{bmatrix}\mathcal{S}^{2}_{1,2,M}& 0_M & 0_{M,2(N-2)M}\\ 0_M & 0_M & 0_{M,2(N-2)M}\end{bmatrix}\\
        \mathcal{B}_{1,M}&=&\begin{bmatrix}0_M & 0_M\\ 0_M & \mathcal{S}^2_{2,1,M}\\ 0_{2(N-2)M,M} & 0_{2(N-2)M,M}\end{bmatrix}\\
        \mathcal{C}_{1,M}&=&\begin{bmatrix} I_M & \mathcal{S}^3_{2,2,M}& \cdots \\ \mathcal{S}^2_{2,2,M} & I_M & \cdots\\ \cdots & \cdots &\cdots \end{bmatrix},
\end{eqnarray*}
where $\widetilde{f}_{1,M}=[f_{2,M}\ f_{3,M}\ \ldots f_{N,M}]^\top$.

{\em Stage j: elimination of the unknowns $f_{j-1,M}$}. Repeating the same steps outlined above we continue the elimination process until we arrive at the following linear system 
\begin{eqnarray*}
  \begin{bmatrix}\mathcal{D}_{j-1,M}& \mathcal{A}_{j-1,M}\\\mathcal{B}_{j-1,M} & \mathcal{C}_{j-1,M} \end{bmatrix}\begin{bmatrix}f_{j-1,M}\\\widetilde{f}_{j-1,M}\end{bmatrix}&=&\begin{bmatrix}\rhs_{j-1,M}\\0_{2(N-j+1)M,1}\end{bmatrix}\\
    \mathcal{D}_{j-1,M}&=&\begin{bmatrix}I_M&\mathcal{S}^{j}_{j-1,j-1,M}\\\mathcal{S}^{top}_{j-1,M} & I_M\end{bmatrix}\\
        \mathcal{A}_{j-1,M}&=&\begin{bmatrix}\mathcal{S}^{j}_{j-1,j,M}& 0_M & 0_{M,2(N-j)M}\\ 0_M & 0_M & 0_{M,2(N-j)M}\end{bmatrix}\\
        \mathcal{B}_{j-1,M}&=&\begin{bmatrix}0_M & 0_M\\ 0_M & \mathcal{S}^j_{j,j-1,M}\\ 0_{2(N-j)M,M} & 0_{2(N-j)M,M}\end{bmatrix}\\
        \mathcal{C}_{j-1,M}&=&\begin{bmatrix} I_M & \mathcal{S}^{j+1}_{j,j,M}& \cdots \\ \mathcal{S}^j_{j,j,M} & I_M & \cdots\\ \cdots & \cdots &\cdots \end{bmatrix},
\end{eqnarray*}
where $\widetilde{f}_{j-1,M}=[f_{j,M}\ f_{j+1,M}\ \ldots f_{N,M}]^\top$. We proceed by eliminating the unknowns $f_{j-1,M}$ from the linear system above. First, we have that
\[
f_{j-1,M}=-\mathcal{D}_{j-1,M}^{-1}A_{j-1,M}\widetilde{f}_{j-1,M}+\mathcal{D}_{j-1,M}^{-1}\rhs_{j-1,M}
\]
and thus
\[
(\mathcal{C}_{j-1,M}-\mathcal{B}_{j-1,M}\mathcal{D}_{j-1,M}^{-1}A_{j-1,M})\widetilde{f}_{j-1,M}=\rhs_{j,M}, \rhs_{j,M}:=-\mathcal{B}_{j-1,M}\mathcal{D}_{j-1,M}^{-1}\rhs_{j-1,M}.
\]
The last linear system can be in turn written in expanded form if we make use of formula~\eqref{inv_matrix_explicit} to compute the inverse of $\mathcal{D}_{j-1,M}$,
\begin{eqnarray*}
  \begin{bmatrix}\mathcal{D}_{j,M}& \mathcal{A}_{j,M}\\\mathcal{B}_{j,M} & \mathcal{C}_{j,M} \end{bmatrix}\begin{bmatrix}f_{j,M}\\\widetilde{f}_{j,M}\end{bmatrix}&=&\begin{bmatrix}\rhs_{j,M}\\0_{2(N-j)M,1}\end{bmatrix}\\
    \rhs_{j,M}&=&-\mathcal{B}_{j-1,M}\mathcal{D}_{j-1,M}^{-1}\rhs_{j-1,M}\\
    \mathcal{D}_{j,M}&=&\begin{bmatrix}I_M&\mathcal{S}^{j+1}_{j,j,M}\\\mathcal{S}^{top}_{j,M} & I_M\end{bmatrix}\\
      \mathcal{S}^{top}_{j,M}&=&\mathcal{S}^j_{j,j-1,M}(I_M-\mathcal{S}^{top}_{j-1,M}\mathcal{S}^j_{j-1,j-1,M})^{-1}\mathcal{S}^{top}_{j-1,M}\mathcal{S}^j_{j-1,j,M}+\mathcal{S}^j_{j,j,M}\\
        \mathcal{A}_{j,M}&=&\begin{bmatrix}\mathcal{S}^{j+1}_{j,j+1,M}& 0_M & 0_{M,2(N-j-1)M}\\ 0_M & 0_M & 0_{M,2(N-j-1)M}\end{bmatrix}\\
        \mathcal{B}_{j,M}&=&\begin{bmatrix}0_M & 0_M\\ 0_M & \mathcal{S}^{j+1}_{j+1,j,M}\\ 0_{2(N-j-1)M,M} & 0_{2(N-j-1)M,M}\end{bmatrix}\\
        \mathcal{C}_{j,M}&=&\begin{bmatrix} I_M & \mathcal{S}^{j+2}_{j+1,j+1,M}& \cdots \\ \mathcal{S}^{j+1}_{j+1,j+1,M} & I_M & \cdots\\ \cdots & \cdots &\cdots \end{bmatrix},
\end{eqnarray*}
where $\widetilde{f}_{j,M}=[f_{j+1,M}\ f_{j+2,M}\ \ldots f_{N,M}]^\top$.

The algorithm ends with a linear system involving only the Robin data $f_{N,M}$ corresponding to the last interface $\Gamma_N$,
\begin{equation}\label{eq:last_eq}
  \begin{bmatrix}I_M & \mathcal{S}^{N+1}_M\\ \mathcal{S}^{top}_{N,M} & I_M\end{bmatrix}f_{N,M}=\rhs_{N,M},\quad  \rhs_{N,M}:=-\mathcal{B}_{N-1,M}\mathcal{D}_{N-1,M}^{-1}\rhs_{N-1,M},
  \end{equation}
which is solved using again formula~\eqref{inv_matrix_explicit} and LU factorizations. Once the discretized Robin data $f_{N,M}$ is computed, all the other discretized Robin data corresponding to the interfaces $\Gamma_j,0\leq j<N$ are computed using backward substitution via the recursions
\[
f_{j-1,M}=-\mathcal{D}_{j-1,M}^{-1}A_{j-1,M}\widetilde{f}_{j-1,M}+\mathcal{D}_{j-1,M}^{-1}\rhs_{j-1,M},1\leq j\leq N,
\]
where we recall that $\widetilde{f}_{j-1,M}=[f_{j,M}\ f_{j+1,M}\ \ldots f_{N,M}]^\top$. In order to streamline the recursions above, we store in the elimination process the quantities $\mathcal{D}_{j-1,M}^{-1}\rhs_{j-1,M}$ and the non-zero blocks of the matrices $\mathcal{D}_{j-1,M}^{-1}A_{j-1,M}$ (for each index $j$ there are only two $M\times M$ such blocks). Thus, the storage require to perform the elimination algorithm followed by the backward substitution process is $\mathcal{O}(2NM^2)$.

We first note that the elimination process presented above consists of $N\!+\!1$ steps, each step requiring (a) the calculation of the Nystr\"om matrices $\mathcal{S}^j_M$ whose cost is $\mathcal{O}(M^3)$, as well as (b) the inversion of a $M\!\times\!M$ matrix for the calculation of $\mathcal{D}^{-1}_{j,M}$ per formula~\eqref{inv_matrix_explicit}, which leads to a total computational cost proportional to $(N+1)M^3$. One drawback of the elimination algorithm presented above is that it is sequential: essentially the Robin data is peeled off layer by layer from the DDM linear system. We are exploring alternative strategies based on hierarchical RtR matrices mergings that could lead to efficient parallelization strategies~\cite{pedneault2017schur}.  Of course, for a large numbers of layers, iterative solvers such as GMRES could also provide an alternative strategy for the solution of the DDM linear system. However, large numbers of layers/subdomains produce large numbers of DDM iterations~\cite{jerez2017multitrace} if classical Robin conditions are used on subdomain interfaces.

\section{Analysis of the Schur complement elimination algorithm\label{schur}}

A natural question is why the elimination procedure described in Section~\ref{nystrom} does not break down. This issue has been explored in a different context in~\cite{pedneault2017schur}, where the elimination process was shown to be equivalent to merging of RtR maps. Indeed, the matrices $\mathcal{S}^{top}_{j,M}$ are themselves Nystr\"om  discretization matrices of the RtR operators
\begin{equation}\label{RtRint}
   \mathcal{S}^{top,j}(\psi_j):=(\gamma_j\partial_{n_j} u+ i\eta\  u)|_{\Gamma_j},
 \end{equation}
 where $u$ is the solution of the following well posed problem:
 \begin{eqnarray*}
   \Delta u+k(x)^2u&=&0\ {\rm in}\ \cup_{\ell=0}^j\Omega_\ell,\quad k(x)=k_\ell,\ x\in\Omega_\ell,\ 0\leq l\leq j\\
  \gamma_j\partial_{n_j} u+ i\eta\  u &=&\psi_j\ {\rm on}\ \Gamma_j,
 \end{eqnarray*}
 with (i) $u$ and $\gamma_\ell\partial_{n_\ell}u$ continuous across $\Gamma_{\ell}$ for $0\leq l<j$; and (ii) $u$ radiating in $\Omega_0$. Thus, the recurrence formula
 \begin{equation}\label{eq:merge}
   \mathcal{S}^{top}_{j,M}=\mathcal{S}^j_{j,j-1,M}(I_M-\mathcal{S}^{top}_{j-1,M}\mathcal{S}^j_{j-1,j-1,M})^{-1}\mathcal{S}^{top}_{j-1,M}\mathcal{S}^j_{j-1,j,M}+\mathcal{S}^j_{j,j,M}, \quad \text{for }\; 1\leq j,
 \end{equation}
 where $\mathcal{S}^{top}_{0,M}:=\mathcal{S}^0_M$, can be viewed as a means to compute Nystr\"om discretization of the RtR map $\mathcal{S}^{top,j}$ defined in equation~\eqref{RtRint} via recursive merging of
 the Nystr\"om discretization matrices of the RtR maps $\mathcal{S}^\ell$. We will establish in this section the invertibility of the operators $I-\mathcal{S}^{top}_{j-1}\mathcal{S}^j_{j-1,j-1}$ in appropriate functional spaces. 
 
The matrices $I_M-\mathcal{S}^{top}_{j-1,M}\mathcal{S}^j_{j-1,j-1,M}$ in~\eqref{eq:merge} are high-order Nystr\"om discretizations of the operators $I-\mathcal{S}^{top}_{j-1}\mathcal{S}^j_{j-1,j-1}$. As such, the invertibility of the former matrices is a consequence of the invertibility of the latter operators owing to the fact that  $\mathcal{S}^{top}_{j-1,M}$ and $\mathcal{S}^j_{j-1,j-1,M}$ can be shown to converge in strong operator norms to $\mathcal{S}^{top}_{j-1}$ and $\mathcal{S}^j_{j-1,j-1}$ as $M\to\infty$~\cite{KressColton}.

 We start by analyzing the invertibility of the operators $I-\mathcal{S}^0\mathcal{S}^{1}_{0,0}$, on which hinges Stage 1 of the elimination algorithm. We will make use of the quasi-periodic Sobolev spaces $H^s_{per}(\Gamma)$ of $\alpha$-quasi-periodic distributions defined on a generic interface $\Gamma$ that is the graph of a periodic function of period $d$. These spaces can be defined in terms of Fourier series. We use the mapping properties of boundary-integral operators associated with quasi-periodic Green functions $G^q_k$ and shifted quasi-periodic Green functions $G_{k,h}^{q,j}$. Since both of these functions have the same singularity as the free space Green function $G_k$ corresponding to the same wavenumber $k$, the mapping properties of the boundary-integral operators associated with quasi-periodic Green functions can be easily derived by simply translating to the periodic setting the classical mapping properties of boundary-integral operators whose domain of integration is a closed curve in $\mathbb{R}^2$. The invertibility of the operator  $I-\mathcal{S}^0\mathcal{S}^{1}_{0,0}$ is established via Fredholm arguments, and it relies on the explicit representations of the operators $\mathcal{S}^0$ and $\mathcal{S}^1$ derived in equations~\eqref{eq:calc_S0} and~\eqref{eq:S_layer_j} in the case when neither $k_0$ nor $k_1$ are Wood frequencies. In the cases when $k_0$ or $k_1$ are Wood frequencies, representations~\eqref{eq:calc_S0_h} and that given in Theorem~\ref{inv_layer_Wood} ought to be used.  Nevertheless, given that the shifted quasi-periodic Green functions have the same singularities as the quasi-periodic Green functions, in the computational domains under considerations, the arguments in the proof of the result below essentially do not change in the Wood-frequency case.  In what follows, we suppose that (1) $k_j<k_{j+1}$ for all $j$ and (2) $\gamma_j=1$ for all $j$ (we will show in Appendix~\ref{wp_proof} that these two assumptions guarantee the well-posedness of the transmission problem under consideration). We establish
 \begin{theorem}\label{Fredholm}
   The operator $I-\mathcal{S}^0\mathcal{S}^{1}_{0,0}:H^{-1/2}_{per}(\Gamma_0)\to H^{1/2}_{per}(\Gamma_0)$ is Fredholm of index 0 under the assumption that $\Gamma_0$ is $C^2$. 
   \end{theorem}
 \begin{proof}
   Let us assume that $k_0$ is not a Wood frequency. We first establish that, given the representation
   \[
   \mathcal{S}^0=I+2Z_0S_{\Gamma_0,k_0}^q\left(\frac{1}{2}I+(K_{\Gamma_0,k_0}^q)^\top-Z_0S_{\Gamma_0,k_0}^q\right)^{-1},
   \]
 the operator $\mathcal{S}^0$ can be expressed in the form
   \begin{equation}\label{eq:f_repr}
     \mathcal{S}^0=I+4Z_0S_{\Gamma_0,k_0+i\varepsilon_0}^q+\mathcal{T}_0,\quad \mathcal{T}_0:H^{-1/2}_{per}(\Gamma_0)\to H_{per}^{3/2}(\Gamma_0),\quad \varepsilon_0>0,\quad Z_0=i\eta.
   \end{equation}
   The decomposition in equation~\eqref{eq:f_repr} can be achieved if we choose the operator $\mathcal{T}_0$ in the following manner:
   \[
   \mathcal{T}_0=4Z_0(S_{\Gamma_0,k_0}^q-S_{\Gamma_0,k_0+i\varepsilon_0}^q)+8Z_0S_{\Gamma_0,k_0}^q\left(\frac{1}{2}I+2(K_{\Gamma_0,k_0}^q)^\top-2Z_0S_{\Gamma_0,k_0}^q\right)^{-1}((K_{\Gamma_0,k_0}^q)^\top-Z_0S_{\Gamma_0,k_0}^q).
   \]
   Now, given that $S_{\Gamma_0,k_0}^q-S_{\Gamma_0,k_0+i\varepsilon_0}:H^{-1/2}_{per}(\Gamma_0)\to H_{per}^{3/2}(\Gamma_0)$, $S_{\Gamma_0,k_0}^q:H^{s}_{per}(\Gamma_0)\to H_{per}^{s+1}(\Gamma_0)$ for $-1/2\leq s\leq 1$, and $(K_{\Gamma_0,k_0}^q)^\top:H^{s}_{per}(\Gamma_0)\to H_{per}^{s+2}(\Gamma_0)$ for $-1/2\leq s\leq 0$, it follows  that $\mathcal{T}_0:H^{-1/2}_{per}(\Gamma_0)\to H_{per}^{3/2}(\Gamma_0)$. The role of the complex wavenumber $k_0+i\varepsilon_0$ in the definition of the quasi-periodic single-layer operators in equation~\eqref{eq:f_repr} will be made clear in what follows.

   We establish next a decomposition of the operator $\mathcal{S}^1_{0,0}$ similar in spirit to that in equation~\eqref{eq:f_repr}. To this end, we revisit representation~\eqref{eq:S_layer_j}, valid for a general periodic layer and a wavenumber~$k$ that is not a Wood frequency. Let us define 
   \[
  \mathcal{A}:= \begin{bmatrix}(1/2)I+(K^q_{k,t,t})^\top +Z S^q_{k,t,t} & (K^q_{k,b,t})^\top+Z S^q_{k,b,t}\\ (K^q_{k,t,b})^\top+Z S^q_{k,t,b} & (1/2)I+(K^q_{k,b,b})^\top +Z S^q_{k,b,b}\end{bmatrix}=\begin{bmatrix}\mathcal{A}_{t,t} & \mathcal{A}_{t,b}\\ \mathcal{A}_{b,t} & \mathcal{A}_{b,b}\end{bmatrix}.
  \]
  We have
  \[
  \mathcal{A}^{-1}=\begin{bmatrix}\mathcal{A}_{t,t}^{-1}+\mathcal{A}_{t,t}^{-1}\mathcal{A}_{t,b}\mathcal{D}\mathcal{A}_{b,t}\mathcal{A}_{t,t}^{-1} & -\mathcal{A}_{t,t}^{-1}\mathcal{A}_{t_b}\mathcal{D}\\ -\mathcal{D}\mathcal{A}_{b,t}\mathcal{A}_{t,t}^{-1} & \mathcal{D}\end{bmatrix}=\begin{bmatrix}\widetilde{\mathcal{A}}_{t,t} & \widetilde{\mathcal{A}}_{t,b}\\ \widetilde{\mathcal{A}}_{b,t} & \widetilde{\mathcal{A}}_{b,b}\end{bmatrix}
  \]
  where
  \[
  \mathcal{D}=(\mathcal{A}_{b,b}-\mathcal{A}_{b,t}\mathcal{A}_{t,t}^{-1}\mathcal{A}_{t,b})^{-1}.
  \]
  The invertibility of the operator $\mathcal{A}_{t,t}$ needed in formulas above can be established similarly to the result in Theorem~\ref{well_p_0}.  The invertibility of the operator $\mathcal{D}$, in turn, can be seen to be equivalent to the invertibility of the matrix operator $\mathcal{A}$. Given that the kernels of the boundary-integral operators that enter in the definition of $\mathcal{A}_{t,b}$ and $\mathcal{A}_{b,t}$ are regular, we have $\mathcal{A}_{t,b}:H^s_{per}(\Gamma_t)\to H_{per}^{s+2}(\Gamma_b)$ and $\mathcal{A}_{b,t}:H^s_{per}(\Gamma_b)\to H_{per}^{s+2}(\Gamma_t)$.   Since
  \[
  \mathcal{S}=\begin{bmatrix}I &0\\0& I\end{bmatrix}+2Z\begin{bmatrix}S^q_{k,t,t} & S^q_{k,b,t}\\ S^q_{k,t,b} & S^q_{k,b,b}\end{bmatrix}\mathcal{A}^{-1}=\begin{bmatrix}\mathcal{S}_{t,t} & \mathcal{S}_{t,b}\\ \mathcal{S}_{b,t} & \mathcal{S}_{b,b}\end{bmatrix},
  \]
  we have
  \begin{align}\label{eq:s_repr}
  \mathcal{S}_{t,t}&=I+2ZS^q_{k,t,t}\widetilde{\mathcal{A}}_{t,t}+2ZS^q_{k,b,t}\widetilde{\mathcal{A}}_{b,t}=I+4ZS^q_{k_0+i\varepsilon_0,t,t}+\widetilde{\mathcal{S}_{t,t}}, \\
  \widetilde{\mathcal{S}_{t,t}}&:H^{-1/2}_{per}(\Gamma_t)\to H_{per}^{3/2}(\Gamma_t)
  \end{align}
taking into account the fact that $S^q_{k,b,t}:H^s_{per}(\Gamma_b)\to H_{per}^{s+2}(\Gamma_t)$. In addition, we obtain
 \begin{align}\label{eq:s_repr_b}
  \mathcal{S}_{b,b}&=I+2ZS^q_{k,t,b}\widetilde{\mathcal{A}}_{t,b}+2ZS^q_{k,b,b}\widetilde{\mathcal{A}}_{b,b}=I+4ZS^q_{k_0+i\varepsilon_0,b,b}+\widetilde{\mathcal{S}_{b,b}}, \\
   \widetilde{\mathcal{S}_{b,b}}&:H^{-1/2}_{per}(\Gamma_b)\to H_{per}^{3/2}(\Gamma_b)
 \end{align}
 as well as the following smoothing  properties of the cross operators $\mathcal{S}_{t,b}$ and $\mathcal{S}_{b,t}$,
 \begin{equation}\label{eq:map_cross}
   \mathcal{S}_{t,b}:H^{-1/2}_{per}(\Gamma_t)\to H_{per}^{3/2}(\Gamma_b),\qquad \mathcal{S}_{b,t}:H^{-1/2}_{per}(\Gamma_b)\to H_{per}^{3/2}(\Gamma_t).
   \end{equation}
Combining the results in~\eqref{eq:f_repr} and~\eqref{eq:s_repr}, we obtain
  \[
  I-\mathcal{S}^0\mathcal{S}^1_{0,0}=-8Z_0S_{\Gamma_0,k_0+i\varepsilon_0}^q+\mathcal{T}_{0,1},\qquad \mathcal{T}_{0,1}:H^{-1/2}_{per}(\Gamma_0)\to H_{per}^{3/2}(\Gamma_0).
  \]
  Classical arguments~\cite{turc2} can be adapted to the periodic setting to establish
  \[
  \Im (S_{\Gamma_0,k_0+i\varepsilon_0}^q\varphi,\overline{\varphi})\geq c_0\|\varphi\|^2_{H^{-1/2}_{per}(\Gamma_0)},\ c_0>0,
  \]
  from which we obtain
  \[
  \Re (-8Z_0S_{\Gamma_0,k_0+i\varepsilon_0}^q\varphi,\overline{\varphi})\geq  8\eta c_0\|\varphi\|^2_{H^{-1/2}_{per}(\Gamma_0)},\ c_0>0, \eta>0.
  \]
  Finally, given that $\mathcal{T}_{0,1}:H^{-1/2}_{per}(\Gamma_0)\to H_{per}^{3/2}(\Gamma_0)$, it follows that the operator $\mathcal{T}_{0,1}:H^{-1/2}_{per}(\Gamma_0)\to H_{per}^{1/2}(\Gamma_0)$ is compact, and thus the operator $I-\mathcal{S}^0\mathcal{S}^1_{0,0}:H^{-1/2}_{per}(\Gamma_0)\to H_{per}^{1/2}(\Gamma_0)$ can be seen to satisfy a G\aa rding inequality. The result of the Theorem is thus established.  The case when $k_0$ and/or $k_1$ are Wood frequencies can be treated similarly. 
 \end{proof}

 We are now in the position to prove
 \begin{theorem}\label{inv}
   The operator $I-\mathcal{S}^0\mathcal{S}^{1}_{0,0}:H^{-1/2}_{per}(\Gamma_0)\to H^{1/2}_{per}(\Gamma_0)$ is invertible with continuous inverse.
 \end{theorem}
 \begin{proof} Owing to the Fredholm alternative, the theorem follows once we establish the injectivity of the operator $I-\mathcal{S}^0\mathcal{S}^{1}_{0,0}$. Let $\varphi\in Ker(I-\mathcal{S}^0\mathcal{S}^1_{0,0})$ and consider the following $\alpha$-quasi-periodic Helmholtz equation
  \begin{eqnarray*}
    \Delta w_1+k_1^2w_1&=&0\qquad{\rm in}\ \Omega_1^{per}\\
    \partial_{n_1}w_1-Z_0w_1&=&\varphi\qquad{\rm on}\ \Gamma_0,\\
    \partial_{n_1}w_1-Z_0w_1&=&0\qquad{\rm on}\ \Gamma_1,
  \end{eqnarray*}
  where $Z_0=i\eta$. Then
  \[
  \mathcal{S}^1_{0,0}\varphi=(\partial_{n_1}w_1+ Z_0w_1)|_{\Gamma_0}.
  \]
  Consider also the $\alpha$-quasi-periodic Helmholtz equation
  \begin{eqnarray*}
    \Delta w_0+k_0^2w_0&=&0\qquad{\rm in}\ \Omega_0\\
    \partial_{n_0}w_0-Z_0w_0&=&\mathcal{S}^1_{0,0}\varphi\qquad{\rm on}\ \Gamma_0,
  \end{eqnarray*}
 with $w_0$ radiating.  Then using the fact that $\mathcal{S}^0\mathcal{S}^1_{0,0}\varphi=\varphi$ on $\Gamma_0$, we obtain
  \[
  \mathcal{S}^0\mathcal{S}^1_{0,0}\varphi=\partial_{n_0}w_0+Z_0w_0=\partial_{n_1}w_1-Z_0w_1\quad{\rm on}\ \Gamma_0.
  \]
  Thus, we have derived the following system of equations on~$\Gamma_0$
  \begin{eqnarray*}
    \partial_{n_0}w_0-Z_0w_0&=&\partial_{n_1}w_1+Z_0w_1\\
    \partial_{n_0}w_0+Z_0w_0&=&\partial_{n_1}w_1-Z_0w_1.
  \end{eqnarray*}
  from which we obtain
  \begin{equation}\label{eq:equality}
  w_0|_{\Gamma_0}=-w_1|_{\Gamma_0},\qquad \partial_{n_0}w_0|_{\Gamma_0}=\partial_{n_1}w_1|_{\Gamma_0}.
  \end{equation}
  Recall from the proof of Theorem~\ref{wp_Omega_0} the following identity for $w_0$:
 \[
\lim_{h\to\infty}\int_{\Omega_{0,h}^{per}}(|\nabla w_0|^2-k_0^2|w_0|^2)dx=\int_{\Gamma_0} \partial_{n_0}w_0\ \overline{w}_0\ ds+id\sum_{r\in\mathbb{Z},\ \beta_{0,r}>0}\beta_{0,r}|C_r^{+}|^2\,.
\]
On the other hand,
\[
\int_{\Omega_{1}^{per}}(|\nabla w_1|^2-k_1^2|w_1|^2)dx=\int_{\Gamma_0} \partial_{n_1}w_1\ \overline{w}_1\ ds+i\eta \int_{\Gamma_1} |w_1|^2\ ds.
\]
Adding the last two identities and taking into account equation~\eqref{eq:equality}, we derive
\[
\lim_{h\to\infty}\int_{\Omega_{0,h}^{per}}(|\nabla w_0|^2-k_0^2|w_0|^2)dx+\int_{\Omega_{1}^{per}}(|\nabla w_1|^2-k_1^2|w_1|^2)dx=id\sum_{r\in\mathbb{Z},\ \beta_{0,r}>0}\beta_{0,r}|C_r^{+}|^2+i\eta \int_{\Gamma_1} |w_1|^2\ ds.
\]
This implies that $w_1=0$ on $\Gamma_1$, and hence $w_1=0$ in $\Omega_1$ by Holmgren's theorem~\cite{Friedman2008}. Using~\eqref{eq:equality} again we obtain that $w_0=0$ and $\partial_{n_0}w_0=0$ on $\Gamma_0$, which, in turn, implies that $w_0=0$ in $\Omega_0$. From this we finally conclude that $\varphi=0$ on $\Gamma_0$.
 \end{proof}

 As a consequence of the results in Theorem~\ref{Fredholm} and Theorem~\ref{inv} we obtain 

 \begin{corollary}\label{mapp_2} The operator $\mathcal{S}^{top}_1$ defined in equation~\eqref{RtRint} can be expressed in the form
   \[
   \mathcal{S}^{top}_1=I+4Z_0S_{\Gamma_1,k_1+i\varepsilon_1}^q+\mathcal{T}_1,\quad \mathcal{T}_1:H^{-1/2}_{per}(\Gamma_1)\to H_{per}^{3/2}(\Gamma_1),\ \varepsilon_1>0.
   \]
 \end{corollary}
 \begin{proof} Consider the following Helmholtz scattering problem with Robin boundary conditions on $\Gamma_1$:  Find $\alpha$-quasi-periodic solutions $u_0$ and $u_1$ such that 
   \begin{eqnarray*}
     \Delta u_0+k_0^2u&=&0\ {\rm in}\ \Omega_0\\
     \Delta u_1+k_1^2u_1&=&0\ {\rm in}\ \Omega_1\\
     u_0&=&u_1\ {\rm on}\ \Gamma_0\\
     \partial_{n_0}u_0&=&-\partial_{n_1}u_1\ {\rm on}\ \Gamma_0\\
  \gamma_1\partial_{n_1} u+ i\eta\  u_1 &=&\psi_1\ {\rm on}\ \Gamma_1,
 \end{eqnarray*}
and $u_0$ is radiating in $\Omega_0$. The arguments in the proof of Theorem~\ref{inv} can be applied to show the well-posedness of the this problem. Reformulating it in terms of matching Robin data on $\Gamma_0$ and applying the same arguments as in Section~\ref{nystrom} but at the operator (continuous) level, we obtain
   \[
   \mathcal{S}^{top}_{1}=\mathcal{S}^1_{1,0}(I-\mathcal{S}^{0}\mathcal{S}^1_{0,0})^{-1}\mathcal{S}^{0}\mathcal{S}^1_{0,1}+\mathcal{S}^1_{1,1}.
   \]
Using the representation above together with the properties recounted in equations~\eqref{eq:s_repr_b} and~\eqref{eq:s_repr_b} with $\Gamma_t=\Gamma_0$, $\Gamma_b=\Gamma_1$, and $Z=Z_0=i\eta$, the result follows.
 \end{proof}

 The procedure presented above can be repeated inductively at the operator level, and can be viewed as means to recursively merge the RtR operators $\mathcal{S}^{top}_{j-1}$ and $\mathcal{S}^j$ in order to obtain the RtR operator $\mathcal{S}^{top}_j$ according to the formula  
 \begin{equation}\label{eq:merge_C}
 \mathcal{S}^{top}_{j}=\mathcal{S}^j_{j,j-1}(I-\mathcal{S}^{top}_{j-1}\mathcal{S}^j_{j-1,j-1})^{-1}\mathcal{S}^{top}_{j-1}\mathcal{S}^j_{j-1,j}+\mathcal{S}^j_{j,j}, 2\leq j.
 \end{equation}
 Equation~\eqref{eq:merge_C} constitutes the continuous analogue of equation~\eqref{eq:merge}. The invertibility of the operators $I-\mathcal{S}^{top}_{j-1}\mathcal{S}^j_{j-1,j-1}$ for $2\leq j\leq N$ can be established similarly to the results in Theorem~\ref{Fredholm} and Theorem~\ref{inv} using the link in Corollary~\ref{mapp_2}. Indeed, it is straightforward to establish by induction that
 \[
 \mathcal{S}^{top}_{j-1}=I+4Z_0S_{\Gamma_{j-1},k_1+i\varepsilon_{j-1}}^q+\mathcal{T}_{j-1},\quad \mathcal{T}_{j-1}:H^{-1/2}_{per}(\Gamma_{j-1})\to H_{per}^{3/2}(\Gamma_{j-1}),\ \varepsilon_{j-1}>0,\ 3\leq j\leq N.
 \]
 We note that this latter representation suffices to establish the Fredholm property of the operator $I-\mathcal{S}^{top}_{j-1}\mathcal{S}^j_{j-1,j-1}$ for all $j: 3\leq j\leq N$; see Theorem~\ref{Fredholm}. The arguments in the proof of Theorem~\ref{inv} also translate almost verbatim to obtain the invertibility of the operators $I-\mathcal{S}^{top}_{j-1}\mathcal{S}^j_{j-1,j-1}$ for all $j:3\leq j\leq N$. It is the very last step in the algorithm when we merge $\mathcal{S}^{top}_N$ and $\mathcal{S}^{N+1}$ that is markedly different on account of the fact that the layer $\Omega_{N+1}$ is semi-infinite and thus the arguments in the proof of Theorem~\ref{inv} have to be modified according to those in Theorem~\ref{thm:uniqueness1} in Appendix~\ref{wp_proof}.
 \begin{theorem}\label{invLast}
   The operator $I-\mathcal{S}^{top}_N\mathcal{S}^{N+1}:H^{-1/2}_{per}(\Gamma_N)\to H^{1/2}_{per}(\Gamma_N)$ is invertible with continuous inverse.
 \end{theorem}
 \begin{proof}  Note that the Fredholm property of the operator $I-\mathcal{S}^{top}_N\mathcal{S}^{N+1}$ essentially follows via the same arguments as in Theorem~\ref{Fredholm}. Owing to the Fredholm alternative, the result follows once we establish the injectivity of the operator $I-\mathcal{S}^{top}_N\mathcal{S}^{N+1}$. Let $\varphi\in Ker(I-\mathcal{S}^{top}_{N}\mathcal{S}^{N+1})$ and consider the following $\alpha$-quasi-periodic Helmholtz equation
  \begin{eqnarray*}
    \Delta w_{N+1}+k_{N+1}^2w_{N+1}&=&0\qquad{\rm in}\ \Omega_{N+1}^{per}\\
    \partial_{n_{N+1}}w_{N+1}-Z_0w_{N+1}&=&\varphi\qquad{\rm on}\ \Gamma_{N}
  \end{eqnarray*}
  with $w_{N+1}$ radiating in $\Omega_{N+1}$ and $Z_0=i\eta$. Then
  \[
  \mathcal{S}^{N+1}\varphi=(\partial_{n_{N+1}}w_{N+1}+ Z_0w_{N+1})|_{\Gamma_N}.
  \]
  Consider also the following $\alpha$-quasi-periodic Helmholtz equation
  \begin{eqnarray*}
    \Delta w+k(x)^2w&=&0\qquad{\rm in}\ \cup_{\ell=0}^N\Omega_\ell, \\
    k(x)&=&k_\ell\qquad {\rm in}\ \Omega_\ell \\
    \partial_{n_N}w-Z_0w&=&\mathcal{S}^{N+1}\varphi\qquad{\rm on}\ \Gamma_{N},
  \end{eqnarray*}
  where $w$ and $\partial_{n_\ell}w$ are continuous across the interfaces $\Gamma_\ell$ for $0\leq \ell \leq N-1$ and $w$ is radiating in the domain $\Omega_0$. We have then
  \[
  \mathcal{S}^{top}_N\mathcal{S}^{N+1}\varphi=\partial_{n_N}w+Z_0w=\partial_{n_{N+1}}w_{N+1}-Z_0w_{N+1}\quad{\rm on}\ \Gamma_N
  \]
  using the fact that $\mathcal{S}^{top}_N\mathcal{S}^{N+1}\varphi=\varphi$ on $\Gamma_N$. Thus
  \begin{equation}\label{eq:equalityN}
  w|_{\Gamma_N}=-w_{N+1}|_{\Gamma_N}\qquad \partial_{n_N}w|_{\Gamma_N}=\partial_{n_{N+1}}w_{N+1}|_{\Gamma_N}.
  \end{equation}
  Applying Green's identities in each domain $\Omega_j^{per}$ for $0\leq j\leq N+1$ and taking into account the continuity conditions~\eqref{eq:equalityN} together with the continuity of $w$ and its normal derivatives across interfaces $\Gamma_j,0\leq j\leq N-1$, we obtain that $C^{+}_r=0$ for all indices $r$ such that $\beta_{0,r}>0$ and $C^{-}_r=0$ for all indices $r$ such that $\beta_{N+1,r}>0$. The proof of the Theorem follows by applying analogous arguments as in the proof of Theorem~\ref{thm:uniqueness1} to $\partial_{x_2}w$ (which is continuous across the interfaces $\Gamma_j$ for $0\leq j\leq N-1$) and $-\partial_{x_2}w_{N+1}$.
 \end{proof}

 We presented in this section an explanation of the fact that the Schur complement elimination process described in Section~\ref{nystrom} does not break down. Incidentally, we obtain as a byproduct of the results in this section a proof of the equivalence between the scattering PDE problem~\eqref{system_t} and its DDM formulation~\eqref{ddm_t_exp} under the assumption that the former is well posed.  According to the Fredholm results established in Theorem~\ref{Fredholm} and Theorem~\ref{invLast}, the DDM formulation~\eqref{ddm_t_exp} \notesps{requires inversions of operators that are compact perturbations of single layer boundary integral operators. As such, the DDM formulation~\eqref{ddm_t_exp} is not particularly suitable to Krylov subspace linear algebra solvers, especially for configurations that involve large numbers of layers.} It is possible to derive DDM formulations that \notesps{are more amenable to Krylov subspace linear algebra solvers} if more general transmission operators are used instead of the multiplicative factors $i\eta$ in the Robin conditions~\cite{boubendir2017domain}. We are investigating such an approach in the context of periodic layered media. 

\section{Numerical results\label{num}}

We present a variety of numerical results regarding transmission scattering problems in periodic layered media. In all cases we assume that all the coefficients $\gamma_j$ that appear in equations~\eqref{system_t} are equal to $1$, and all the numerical results presented are at normal incidence. Qualitatively similar results are obtained in the case of general $\gamma_j$ and oblique incidence. We present two error indicators, one concerning the energy balance and one concerning errors in the Rayleigh coefficient $B_0^{+}$ in expansion~\eqref{eq:rad_up}. The energy conservation defect is defined as
\begin{equation}\label{eq:en}
  \varepsilon_{en}=\left|\sum_{r\in U^{+}}\frac{\beta_{0,r}}{\beta_0}|C_r^{+}|^2+\sum_{r\in U^{-}}\frac{\beta_{N+1,r}}{\beta_0}|C_r^{-}|^2-1\right|
\end{equation}
where $U^{+}:=\{r\in\mathbb{Z}:\beta_{0,r}\geq 0\}$ and $U^{-}:=\{r\in\mathbb{Z}:\beta_{N+1,r}\geq 0\}$. In all the numerical tests presented in this section, the energy conservation defect turned out to be indicative of the errors achieved in the Rayleigh coefficients of the scattered/transmitted fields. We also denote by $\varepsilon_1$ the relative error achieved in the Rayleigh coefficient $C_0^{+}$, measured against a reference solution that was produced with refined discretizations and large enough values of the parameter $A$ in the windowed quasi-periodic functions~\eqref{eq:qper_GA} and~\eqref{eq:qper_G_shiftA} respectively. This parameter is chosen to be large enough so that very small energy conservation defects were achieved.  \notesps{We do not know of a theoretical way to determine the parameter $A$ in the various layers that optimizes the balance between accuracy and efficiency. In practice, this parameter is selected using information gathered from numerical experiments.  Particularly, we choose values of $A$ large enough so that the RtR discretizations $\mathcal{S}^j_M$ have norms as close to 1 as possible.}

At the heart of our DDM algorithm are computations of RtR maps, which rely on evaluations of boundary-integral operators involving quasi-periodic Green functions. The quasi-periodic functions are approximated via windowing functions cf.~\eqref{eq:qper_GA} and~\eqref{eq:qper_G_shiftA}. We discuss the selection of various parameters that enter the definition of the windowed Green function defined in equation~\eqref{eq:qper_GA} and the shifted windowed Green function defined in equation~\eqref{eq:qper_G_shiftA}. In cases when the wavenumbers $k_\ell$ are not Wood frequencies we had to choose only the parameter $A$ in the definition of the windowed Green function $G_{k_\ell}^{q,A}$ defined in equation~\eqref{eq:qper_GA}; the windowed Green function $G_{k_\ell}^{q,A}$ converge superalgebraically to the quasi-periodic Green function $G_{k_\ell}^q$ as $A\to\infty$~\cite{bruno2016superalgebraically}. On the other hand, in cases when $k_\ell$ is a Wood frequency, we had to select  two additional parameters in the definition of $G_{k_\ell,h_\ell}^{q,j,A}$~\eqref{eq:qper_G_shiftA}: the number $j$ of shifts and the value of the shift $h_\ell$. The rate of convergence of the Green functions $G_{k_\ell,h_\ell}^{q,j,A}$ is algebraic in $\ell$ as $A\to\infty$~\cite{Delourme}. Naively, the cost of evaluating $G_{k_\ell,h_\ell}^{q,j,A}$ is $j+1$ times more expensive than that of evaluating $G_{k_\ell}^{q,A}$ for a fixed value of $A$.  However, the quantities $G_{k_\ell,h_\ell}^{q,j,A}$ can be evaluated at considerably reduced costs via accurate asymptotic expansions~\cite{Delourme}. The values of the shifts $h_\ell$ should be such that the poles of the Green function $G_{k_\ell,h_\ell}^{q,j,A}$ are outside the computational domain $\Omega_\ell$; this requirement entails that $h_0>0$ and $h_{N+1}<0$, whereas the shifts $h_\ell,1\leq \ell\leq N$ should be positive and larger than the width of the corresponding layer domain $\Omega_\ell$ measured in the $x_2$ direction. In addition to these requirements on the shifts $h_\ell$, there are discrete sets of values of the shifts for which the shifted Green functions $G_{k_\ell,h_\ell}^{q,j,A}$ do not converge, and those sets can be explicitly computed~\cite{Delourme}. Indeed, as explained in the proof of Theorem~\ref{well posedness_Wood}, in the domain $\Omega_\ell$ the forbidden set of shifts consists of values $h_\ell$ such that $e^{i\beta_r h_\ell}= 1$ for an index $r\in U_\ell$, where $U_\ell$ is the set of propagating modes corresponding to the wavenumber $h_\ell$. In practice it is straightforward to choose the shifts so that all the requirements specified above are met.
\notesps{Of course, one can use a fixed number of shifts in the shifted Green function for all layers at all frequencies, and this is guaranteed to work as long as the discrete set of shifts $h$ is avoided.  This set is easy to compute for a given structure.   If one wishes to optimize the computational performance, the Wood frequencies for each layer can be computed a priori to decide whether the shifted Green function or simply the smoothly windowed one should be used.}

After selection of the parameters that enter the various windowed Green functions, the DDM algorithm is implemented according to its description in Section~\ref{nystrom}. We detail in all the numerical experiments the size $M$ of the discretization points used to approximate each Robin data $f_j$. The DDM linear system consists of $2(N+1)M$ unknowns. We mention that it is also possible to use non-conforming discretizations of Robin data, that is to use different $M_j$ for each layer $\Omega_j$; the values of $M_j$ are chosen to resolve the wavenumbers $k_j$ as well as the profiles $\Gamma_j$~\cite{boubendir2017domain}--see Table~\ref{comp7a}.  The Schur complement elimination algorithm described in Section~\ref{nystrom} allows for solution of large DDM linear systems using only limited memory storage. For instance, the numerical experiments in Table~\ref{comp8} and Table~\ref{comp83d} involving DDM linear systems with $81920$ and respectively $163840$ unknowns were run on a MacBookPro machine with 8Gb of memory. An important drawback of the elimination algorithm described in Section~\ref{nystrom} is its sequential nature.

We organize the presentation of the numerical experiments into four categories.  First we treat the case of one interface $\Gamma_0$ which is important in its own right for several applications; then we present results for large numbers of layers; we continue with results involving periodic layers that contain periodic inclusions, as such configurations are relevant to photonics applications; and we conclude with three-dimensional results. We emphasize Wood frequencies to demonstrate the versatility of the shifted Green function method, as such cases are computationally more challenging. We indicate in the headings of each Table whether the windowed Green function $G_k^{q,A}$ or the shifted Green function $G_{k,h}^{q,j,A}$ were used in the numerical experiments. 

{\em Results for periodic transmission problems with one grating interface}. We start with an illustration in Table~\ref{comp1} of the high-order accuracy that can be achieved by the DDM solver in the case of one periodic interface/grating $\Gamma_0$ given by the graph of the $2\pi$ periodic function
$x_2=H/2\cos{x_1}$ for two values of the height $H=0.6$ and $H=2$; in this example \notesps{we took $k_0=4.1$ and $k_1=16.1$ and thus} the period of the interface $\Gamma_0$ is respectively $4.1$ and $16.1$ wavelength across.

\begin{table}
   \begin{center}
     \resizebox{!}{1.2cm}
{   
\begin{tabular}{|c|c|c|c|c|c|c|}
\hline
\multicolumn{3}{|c|} {$G_{k_\ell}^{q,A}, k_0=4.1, k_1=16.1, H=0.6$} & \multicolumn{3}{c|} {$G_{k_\ell}^{q,A}, k_0=4.1, k_1=16.1, H=2$} \\
\cline{1-6}
A & $\varepsilon_{en}$ & $\varepsilon_1$ & A & $\varepsilon_{en}$ & $\varepsilon_1$ \\
\hline
20 & 2.4 $\times$ $10^{-5}$ & 3.4 $\times$ $10^{-6}$ & 40 & 3.5 $\times$ $10^{-4}$ & 8.1 $\times$ $10^{-5}$\\
40 & 3.0 $\times$ $10^{-7}$ & 1.2 $\times$ $10^{-7}$ & 120 & 5.6 $\times$ $10^{-5}$ & 8.2 $\times$ $10^{-6}$\\
80 & 6.1 $\times$ $10^{-8}$ & 1.9 $\times$ $10^{-8}$ & 240 & 8.4 $\times$ $10^{-7}$ & 2.0 $\times$ $10^{-6}$\\
\hline
\end{tabular}
}
\caption{Convergence of the DDM transmission solver in the case of one interface of material discontinuity $\Gamma_0$ given by the grating profile $x_2=H/2\cos{x_1}$ under normal incidence, with wavenumbers $k_0=4.1$ and $k_1=16.1$, and $M=64$. The reference solutions were computed using (1) $A=240$ in the case $H=0.6$ ---with a corresponding $\varepsilon_{en}=5.9\times 10^{-9}$ and $A=400$ in the case $H=2$---with a corresponding $\varepsilon_{en}=3.5\times 10^{-8}$, and $M=128$ in both cases.\label{comp1}}
\end{center}
\end{table}

We continue in Table~\ref{comp2} and Table~\ref{comp3} with numerical results concerning the convergence of the DDM algorithm in the case of one grating profile $\Gamma_0$ given by the graph of the $2\pi$ periodic function $x_2=H/2\cos{x_1}, H=0.6$ for two values of the wavenumbers $k_0$ and $k_1$ that are simultaneously Wood frequencies.  In Table~\ref{comp2} we consider medium frequencies---the set $U^{+}:=\{r\in\mathbb{Z}:\beta_{0,r}\geq 0\}$ consists of 17 propagating modes and the set $U^{-}:=\{r\in\mathbb{Z}:\beta_{N+1,r}\geq 0\}$ consists of 65 propagating modes; in Table~\ref{comp3} higher frequencies---the set $U^{+}:=\{r\in\mathbb{Z}:\beta_{0,r}\geq 0\}$ consists of 31 propagating modes and the set $U^{-}:=\{r\in\mathbb{Z}:\beta_{N+1,r}\geq 0\}$ consists of 121 propagating modes. In the configuration in Table~\ref{comp2} the period of the interface $\Gamma_0$ is 8 and 32 wavelengths across on each side, whereas in the configuration in Table~\ref{comp3} the period of the interface $\Gamma_0$ is 15 and 60 wavelengths across on each side.

\begin{table}
   \begin{center}
     \resizebox{!}{1.2cm}
{   
\begin{tabular}{|c|c|c|c|c|c|}
\hline
\multicolumn{3}{|c|} {$G_{k_\ell,h}^{q,j,A}, j=3, h_0=-h_1= 1.3, k_0=8, k_1=32, H=0.6$} & \multicolumn{3}{c|} {$G_{k_\ell,h}^{q,j,A}, j=5, h_0=-h_1=1.3, k_0=8, k_1=32, H=0.6$} \\
\cline{1-6}
A & $\varepsilon_{en}$ & $\varepsilon_1$ & A & $\varepsilon_{en}$ & $\varepsilon_1$ \\
\hline
20 & 1.5 $\times$ $10^{-3}$ & 5.7 $\times$ $10^{-4}$ & 20 & 7.9 $\times$ $10^{-4}$ & 5.1 $\times$ $10^{-4}$\\
40 & 7.4 $\times$ $10^{-4}$ & 1.2 $\times$ $10^{-4}$ & 40 & 1.5 $\times$ $10^{-4}$ & 7.5 $\times$ $10^{-5}$\\
80 & 1.3 $\times$ $10^{-4}$ & 2.2 $\times$ $10^{-5}$ & 80 & 3.0 $\times$ $10^{-6}$ & 1.3 $\times$ $10^{-5}$\\
120 & 4.1 $\times$ $10^{-5}$ & 7.2 $\times$ $10^{-6}$ & 120 & 9.7 $\times$ $10^{-8}$ & 1.5 $\times$ $10^{-6}$\\
\hline
\end{tabular}
}
\caption{Convergence of the DDM transmission solver in the case of one interface of material discontinuity $\Gamma_0$ given by the grating profile $x_2=H/2\cos{x_1}, H=0.6$, under normal incidence, with $k_0=8$ and $k_1=32$, and $M=128$. In this case both wavenumbers $k_0$ and $k_1$ are Wood frequencies. The reference solutions were computed using $A=240$, $j=5$ and shifts $h_0=-h_1=1.3$ with a corresponding $\varepsilon_{en}=1.6\times 10^{-11}$.\label{comp2}}
\end{center}
\end{table}

\begin{table}
   \begin{center}
     \resizebox{!}{1.2cm}
{   
\begin{tabular}{|c|c|c|c|c|c|}
\hline
\multicolumn{3}{|c|} {$G_{k_\ell,h}^{q,j,A}, j=3, h_0= -h_1=0.3, k_0=15, k_1=60, H=0.6$} & \multicolumn{3}{c|} {$G_{k_\ell,h}^{q,j,A}, j=5, h_0=-h_1=0.3, k_0=15, k_1=60, H=0.6$} \\
\cline{1-6}
A & $\varepsilon_{en}$ & $\varepsilon_1$ & A & $\varepsilon_{en}$ & $\varepsilon_1$ \\
\hline
20 & 4.8 $\times$ $10^{-4}$ & 4.6 $\times$ $10^{-5}$ & 20 & 1.3 $\times$ $10^{-5}$ & 2.7 $\times$ $10^{-5}$\\
40 & 5.0 $\times$ $10^{-5}$ & 8.0 $\times$ $10^{-6}$ & 40 & 3.0 $\times$ $10^{-7}$ & 2.1 $\times$ $10^{-6}$\\
80 & 5.5 $\times$ $10^{-6}$ & 1.3 $\times$ $10^{-6}$ & 80 & 2.8 $\times$ $10^{-8}$ & 1.6 $\times$ $10^{-7}$\\
\hline
\end{tabular}
}
\caption{Convergence of the DDM transmission solver in the case of one interface of material discontinuity $\Gamma_0$ given by the grating profile $x_2=H/2\cos{x_1}, H=0.6$, under normal incidence, with higher frequency wavenumbers $k_0=15$ and $k_1=60$, and $M=256$. In this case both wavenumbers $k_0$ and $k_1$ are Wood frequencies. The reference solutions were computed using $A=240$, $j=5$, $h_0=-h_1=0.3$, and $M=256$ with a corresponding $\varepsilon_{en}=2.2\times 10^{-10}$.\label{comp3}}
\end{center}
\end{table}

In Table~\ref{comp4} we present results concerning a deep grating profile $\Gamma_0$ given by the graph of the $2\pi$ periodic function $x_2=H/2\cos{x_1}, H=2$ for two values of the wavenumbers $k_0$ and $k_1$ that are simultaneously Wood frequencies, that is $k_0=4$ and $k_1=16$. We note that high-accuracy results can be achieved in this case by increasing the size of discretization. 
\begin{table}
   \begin{center}
     \resizebox{!}{1.2cm}
{   
\begin{tabular}{|c|c|c|c|c|c|}
\hline
\multicolumn{3}{|c|} {$G_{k_\ell,h}^{q,j,A}, j=3, h_0=-h_1=0.21, k_0=4, k_1=16, H=2$} & \multicolumn{3}{c|} {$G_{k_\ell,h}^{q,j,A}, j=5, h_0=-h_1=0.21, k_0=4, k_1=16, H=2$} \\
\cline{1-6}
A & $\varepsilon_{en}$ & $\varepsilon_1$ & A & $\varepsilon_{en}$ & $\varepsilon_1$ \\
\hline
20 & 1.7 $\times$ $10^{-4}$ & 3.0 $\times$ $10^{-5}$ & 20 & 3.1 $\times$ $10^{-6}$ & 2.8 $\times$ $10^{-6}$\\
40 & 2.5 $\times$ $10^{-5}$ & 5.1 $\times$ $10^{-6}$ & 40 & 1.5 $\times$ $10^{-7}$ & 2.4 $\times$ $10^{-7}$\\
80 & 3.6 $\times$ $10^{-6}$ & 9.3 $\times$ $10^{-7}$ & 80 & 1.4 $\times$ $10^{-8}$ & 2.0 $\times$ $10^{-8}$\\
\hline
\end{tabular}
}
\caption{Convergence of the DDM transmission solver in the case of one interface of material discontinuity $\Gamma_0$ given by the grating profile $x_2=H/2\cos{x_1}, H=2$, normal incidence, with $k_0=4$ and $k_1=16$, and $M=192$. In this case both wavenumbers $k_0$ and $k_1$ are Wood frequencies. The reference solutions were computed using $A=240$, $j=5$, $h_0=-h_1=0.21$, and $M=256$ with a corresponding $\varepsilon_{en}=2.2\times 10^{-10}$.\label{comp4}}
\end{center}
\end{table}

We conclude the numerical results in this part with a case in Table~\ref{comp5} with a transmission experiment involving one interface of material discontinuity $\Gamma_0$ and two wavenumbers such that one of them is not a Wood frequency while the other is a Wood frequency. 
\begin{table}
   \begin{center}
     \resizebox{!}{1.0cm}
{   
\begin{tabular}{|c|c|c|c|c|c|}
\hline
\multicolumn{3}{|c|} {$G_{k_0}^{q,A}, G_{k_1,h}^{q,j,A}, j=3, h_0=-h_1= 0.3, k_0=4.1, k_1=16, H=0.6$} & \multicolumn{3}{c|} {$G_{k_0}^{q,AA}, G_{k_1,h}^{q,j,A}, j=5, h_0=-h_1=0.3, k_0=4.1, k_1=16, H=0.6$} \\
\cline{1-6}
$A$ & $\varepsilon_{en}$ & $\varepsilon_1$ & $A$ & $\varepsilon_{en}$ & $\varepsilon_1$ \\
\hline
20 & 1.3 $\times$ $10^{-3}$ & 5.4 $\times$ $10^{-4}$ & 20 & 5.8 $\times$ $10^{-6}$ & 3.4 $\times$ $10^{-6}$\\
40 & 4.4 $\times$ $10^{-4}$ & 9.8 $\times$ $10^{-5}$ & 40 & 2.1 $\times$ $10^{-7}$ & 2.2 $\times$ $10^{-7}$\\
80 & 8.6 $\times$ $10^{-5}$ & 1.7 $\times$ $10^{-5}$ & 80 & 2.0 $\times$ $10^{-8}$ & 2.3 $\times$ $10^{-8}$\\
\hline
\end{tabular}
}
\caption{Convergence of the DDM transmission solver in the case of one interface of material discontinuity $\Gamma_0$ given by the grating profile $x_2=H/2\cos{x_1}, H=0.6$, normal incidence, with $k_0=4.1$ and $k_1=16$, and $M=64$. In this case $k_0$ is not a Wood frequency and $k_1$ is a Wood frequencies. The reference solutions were computed using $A=240$, $j=5$, $h_0=-h_1=0.3$, and $M=128$ with a corresponding $\varepsilon_{en}=1.7\times 10^{-9}$.\label{comp5}}
\end{center}
\end{table}

{\em Multiple layers}. The next set of results concern transmission experiments involving multiple periodic layers. In the example that follows we consider the first profile $\Gamma_0$ described by (a) $x_2=F_0(x_1), F_0(x_1):=H/2\cos{x_1}$ and (b) $x_2=F_0(x_1), F_0(x_1):=\pi\ H(0.4\cos(x_1)-0.2\cos(2x_1)+0.4\cos(3x_1))$ and the subsequent profiles $\Gamma_\ell$ being simple down shifted versions of the first profile, that is the grating $\Gamma_\ell$ is given by $x_2=F_\ell(x_1), F_\ell(x_1):=-\ell L + F_0(x_1), 0\leq \ell\leq N$. The first set of results in Table~\ref{comp6} concerns the convergence of the DDM transmission solver in layered configurations consisting of 4 layers (that is $N=2$) separated by interfaces $\Gamma_\ell,0\leq \ell\leq 2$, when each wavenumber $k_\ell, 0\leq\ell\leq 3$ is a Wood frequency. 

\begin{table}
   \begin{center}
     \resizebox{!}{1.8cm}
{   
\begin{tabular}{|c|c|c|c|c|c|}
\hline
\multicolumn{3}{|c|} {$G_{k_\ell,h_\ell}^{q,j,A}, j=3$} & \multicolumn{3}{c|} {$G_{k_\ell,h_\ell}^{q,j,A}, j=5$} \\
\cline{1-6}
$A$ & $\varepsilon_{en}$ & $\varepsilon_1$ & $A$ & $\varepsilon_{en}$ & $\varepsilon_1$ \\
\hline
20 & 7.0 $\times$ $10^{-1}$ & 7.2 $\times$ $10^{-2}$ & 20 & 7.0 $\times$ $10^{-2}$ & 7.2 $\times$ $10^{-2}$\\
40 & 8.7 $\times$ $10^{-4}$ & 1.0 $\times$ $10^{-3}$ & 40 & 8.5 $\times$ $10^{-4}$ & 1.1 $\times$ $10^{-3}$\\
80 & 1.0 $\times$ $10^{-4}$ & 7.2 $\times$ $10^{-5}$ & 80 & 2.7 $\times$ $10^{-5}$ & 3.1 $\times$ $10^{-5}$\\
\hline
\hline
20 & 9.3 $\times$ $10^{-1}$ & 2.4 $\times$ $10^{-2}$ & 20 & 9.4 $\times$ $10^{-1}$ & 2.4 $\times$ $10^{-2}$\\
40 & 2.6 $\times$ $10^{-3}$ & 1.4 $\times$ $10^{-3}$ & 40 & 2.7 $\times$ $10^{-3}$ & 1.5 $\times$ $10^{-3}$\\
80 & 1.9 $\times$ $10^{-4}$ & 1.6 $\times$ $10^{-4}$ & 80 & 1.9 $\times$ $10^{-6}$ & 7.3 $\times$ $10^{-5}$\\
\hline
\end{tabular}
}
\caption{Convergence of the DDM transmission solver in the case of a periodic configuration consisting of 4 layers (that is $N=2$), where the interfaces $\Gamma_\ell,0\leq \ell\leq 2$ are given by grating profiles $F_\ell(x_1)=-\ell L+H/2\cos{x_1}, H=0.6, L=1.3, 0\leq \ell\leq 2$---top panel and $F_\ell(x_1)=-\ell L+\pi\ H(0.4\cos(x_1)-0.2\cos(2x_1)+0.4\cos(3x_1)), H=0.1, L=1.3$---bottom panel, under normal incidence, with $k_\ell=\ell+1$ for $0\leq \ell\leq 3$, and $M=64$. All of the wavenumbers $k_\ell$ are Wood frequencies. The shifts were chosen $h_0=0.3$, $h_1=h_2=2.7$, and $h_3=-0.3$. The reference solutions were computed using $A=120$, $j=5$, and $M=128$ with a corresponding $\varepsilon_{en}=2.6\times 10^{-6}$ (top) and $A=120$, $j=5$, and $M=128$ with a corresponding $\varepsilon_{en}=2.9\times 10^{-6}$ (bottom).\label{comp6}}
\end{center}
\end{table}

In the next set of results in Table~\ref{comp7} we present numerical experiments concerning periodic configurations that involve large numbers of layers \notesps{(i.e 10, 20, and 40 layers)} and associated wavenumbers that are all Wood frequencies. We used shifted Green functions with a number $j=3$ of shifts, as this choice leads to small energy conservation defects. In the examples presented in Table~\ref{comp7}, the interface $\Gamma_j$ is $j+1$ wavelengths across, leading thus to problems that overall are 55, 210, and respectively 820 wavelengths in size.

\begin{table}
   \begin{center}
     \resizebox{!}{1.8cm}
{   
\begin{tabular}{|c|c|c|c|c|c|c|c|c|}
\hline
\multicolumn{3}{|c|} {$G_{k_\ell,h_\ell}^{q,3,A}, N=9$} & \multicolumn{3}{c|} {$G_{k_\ell,h_\ell}^{q,3,A}, N=19$} & \multicolumn{3}{c|} {$G_{k_\ell,h_\ell}^{q,3,A}, N=39$}\\
\cline{1-9}
$A$ & $M$ & $\varepsilon_{en}$ & $A$ & $M$ & $\varepsilon_{en}$ & $A$ & M & $\varepsilon_{en}$ \\
\hline
40 & 64 & 4.1 $\times$ $10^{-3}$ & 40 & 128 & 3.6 $\times$ $10^{-3}$ & 40 & 192  & 5.8 $\times$ $10^{-3}$ \\
80 & 64 & 1.2 $\times$ $10^{-3}$ & 80 & 128 & 8.2 $\times$ $10^{-4}$ & 80 & 192  & 1.9 $\times$ $10^{-3}$ \\
\hline
\hline
80 & 128 & 1.3 $\times$ $10^{-3}$ & 80 & 192 & 7.6 $\times$ $10^{-2}$ & 80 &   256 & 4.7 $\times$ $10^{-2}$ \\
120 & 128 & 2.9 $\times$ $10^{-4}$ & 120 & 192 & 3.1 $\times$ $10^{-4}$ & 120 & 256  & 4.3 $\times$ $10^{-4}$ \\
\hline
\end{tabular}
}
\caption{Energy defect errors produced by the DDM transmission solver for configurations consisting of $N+2$ layers for various values of $N$, where the interfaces $\Gamma_\ell,0\leq \ell\leq N$ are given by grating profiles $F_\ell(x_1)=-\ell L+H/2\cos{x_1}, H=0.6, L=1.3, 0\leq \ell\leq N$ (top panel) and $F_\ell(x_1)=-\ell L+\pi\ H(0.4\cos(x_1)-0.2\cos(2x_1)+0.4\cos(3x_1)), H=0.1, L=1.3, 0\leq\ell\leq N$ (bottom panel), under normal incidence, with $k_\ell=\ell+1$ for $0\leq \ell\leq N+1$, and various values of the discretization size $M$. All of the wavenumbers $k_\ell$ are Wood frequencies. The shifts were chosen $h_0=0.3$, $h_\ell=2.7,1\leq \ell\leq N$, and $h_{N+1}=-0.3$. The discrete DDM linear system has in each case $1280$, $5120$, and respectively $15360$ unknowns and is solved via the Schur complement elimination procedure.\label{comp7}}
\end{center}
\end{table}

In applications that involve high-contrast layer media, using non-conforming DDM discretizations leads to more efficient solvers. We present experiments in Table~\ref{comp7a} concerning configurations consisting of layers with alternating high-contrast material properties. In such settings it is natural to use coarser discretizations to compute the RtR maps corresponding to layers with smaller wavenumbers as well as restriction/interpolation Fourier matrices to match non-conforming interface Robin data. 

\begin{table}
   \begin{center}
     \resizebox{!}{1.2cm}
{   
\begin{tabular}{|c|c|c|c|c|c|c|c|c|}
\hline
\multicolumn{3}{|c|} {$G_{k_\ell}^{q,A}, N=9$} & \multicolumn{3}{c|} {$G_{k_\ell}^{q,A}, N=19$} & \multicolumn{3}{c|} {$G_{k_\ell}^{q,A}, N=39$}\\
\cline{1-9}
$A$ & $M_1/M_2$ & $\varepsilon_{en}$ & $A$ & $M_1,M_2$ & $\varepsilon_{en}$ & $A$ & $M_1,M_2$ & $\varepsilon_{en}$ \\
\hline
30 & 48/96 & 1.0 $\times$ $10^{-2}$ & 30 & 48/96 & 1.8 $\times$ $10^{-2}$ & 30 & 48/96  & 1.6 $\times$ $10^{-2}$ \\
30 & 64/128 & 1.8 $\times$ $10^{-3}$ & 30 & 64/128 & 1.6 $\times$ $10^{-3}$ & 30 & 64/128  & 2.0 $\times$ $10^{-3}$ \\
\hline
\hline
30 & 64/192 & 1.1 $\times$ $10^{-1}$ & 30 & 64/192 & 4.1 $\times$ $10^{-1}$ & 30 & 64/192  & 3.9 $\times$ $10^{-1}$ \\
30 & 96/256 & 1.5 $\times$ $10^{-3}$ & 30 & 96/256 & 8.0 $\times$ $10^{-3}$ & 30 & 96/256  & 8.1 $\times$ $10^{-3}$ \\
\hline
\end{tabular}
}
\caption{Energy defect errors produced by the DDM transmission solver for configurations consisting of $N+2$ layers for various values of $N$, where the interfaces $\Gamma_\ell,0\leq \ell\leq N$ are given by grating profiles $F_\ell(x_1)=-\ell L+H/2\cos{x_1}, H=2, L=0.3, 0\leq \ell\leq N$ (top panel) and $F_\ell(x_1)=-\ell L+\pi\ H(0.4\cos(x_1)-0.2\cos(2x_1)+0.4\cos(3x_1)), H=1, L=0.3, 0\leq \ell\leq N$ (bottom panel), under normal incidence, with $k_\ell=4.2$ for $\ell$ even and $k_\ell=16.2$ for $\ell$ odd, and various values of the non-conformal \notesps{interface} discretization size $M_1$ and $M_2$. For the numerical experiments presented in the top panel the discrete DDM linear system has in each case $1440$, $2880$, and respectively $5760$ unknowns for the coarser discretizations and respectively $1920$, $3840$ and $7680$ for the finer discretizations.  For the numerical experiments presented in the top panel the discrete DDM linear system has in each case $2560$, $5120$, and respectively $10240$ unknowns for the coarser discretizations and respectively $3520$, $7040$ and $14080$ for the finer discretizations. In each case the DDM linear system is solved via the Schur complement elimination procedure.\label{comp7a}}
\end{center}
\end{table}

In the last set of results in this part we present in Table~\ref{comp8} numerical experiments concerning very large numbers of layers and associated wavenumbers that are not Wood frequencies. In such cases, the Schur complement elimination algorithm for the solution of the DDM algorithm reduces the memory requirements via the forward/backward domain sweep. In the examples presented in Table~\ref{comp8}, the interface $\Gamma_j$ is approximately $j+1$ wavelengths across, leading thus to problems that overall are about 820 and respectively 3240 wavelengths in size. \notesps{We mention that all the matrices $\mathcal{D}_{j,M}$ (see Section~\ref{nystrom}) that need be inverted in the Schur complement solution of the problems presented in Table~\ref{comp8} are well conditioned, with condition numbers in the interval $[19,400]$; also, the condition numbers of the matrices $\mathcal{D}_{j,M}$ grow with the layer index $j$. However, the condition numbers of the matrices $\mathcal{D}_{j,M}$ grow with the size $M$ of the discretization cf. Theorem~\ref{Fredholm} and Theorem~\ref{invLast}, yet not drastically. On the other hand, if the solution of  the DDM corresponding to transmission problems with $N=9,19,29$ layers and wavenumbers $k_\ell=\ell+1.2, 0\leq \ell\leq N+1$ were attempted via iterative solvers, the numbers of GMRES iterations required to reach a relative residual of $10^{-4}$ are $270,718,1292$ for the grating interfaces $F_\ell(x_1)=-\ell L+H/2\cos{x_1}, H=2, L=0.3, 0\leq \ell\leq N$, and respectively $318,814,1498$ for the grating interfaces $F_\ell(x_1)=-\ell L+\pi\ H(0.4\cos(x_1)-0.2\cos(2x_1)+0.4\cos(3x_1)), H=1, L=0.3, 0\leq \ell\leq N$ when the discretized RtR matrices are of size $M^2, M=192$ and $\eta=1$. Despite the large numbers of GMRES iterations required for DDM convergence in the experiments above, the condition numbers of the DDM matrices (which can be built in these cases) are reasonable: $99.1, 317.6, 671.4$ in the first case, and respectively $119.2, 427.1, 947.6$ in the second case.}

\begin{table}
   \begin{center}
\begin{tabular}{|c|c|c|c|}
\hline
\multicolumn{2}{|c|} {$G_{k_\ell}^{q,A}, N=39, k_\ell=\ell+1.2, 0\leq \ell\leq 40$} & \multicolumn{2}{c|} {$G_{k_\ell}^{q,A}, N=79, k_\ell=\ell+1.2, 0\leq \ell\leq 80$}\\
\cline{1-4}
A & $\varepsilon_{en}$ & A & $\varepsilon_{en}$ \\
\hline
20 & 5.4 $\times$ $10^{-2}$ & 20 & 5.5 $\times$ $10^{-1}$ \\
40 & 1.3 $\times$ $10^{-3}$ & 40 & 2.3 $\times$ $10^{-3}$ \\
80 & 2.1 $\times$ $10^{-4}$ & 80 & 4.5 $\times$ $10^{-4}$ \\
\hline
\hline
20 & 6.1 $\times$ $10^{-2}$ & 20 & 7.1 $\times$ $10^{-1}$ \\
40 & 1.1 $\times$ $10^{-3}$ & 40 & 1.0 $\times$ $10^{-2}$ \\
80 & 9.8 $\times$ $10^{-5}$ & 80 & 2.4 $\times$ $10^{-3}$ \\
\hline
\end{tabular}
\caption{Convergence of the DDM transmission solver for configuration consisting of $N+2$ layers for various values of $N$, where the interfaces $\Gamma_\ell,0\leq \ell\leq N$ are given by grating profiles $F_\ell(x_1)=-\ell L+H/2\cos{x_1}, H=2, L=0.3, 0\leq \ell\leq N$ (top panel) and $F_\ell(x_1)=-\ell L+\pi\ H(0.4\cos(x_1)-0.2\cos(2x_1)+0.4\cos(3x_1)), H=1, L=0.3, 0\leq \ell\leq N$ (bottom panel), under normal incidence, with $k_\ell=\ell+1.2$ for $0\leq \ell\leq N+1$, and various values of the discretization $M$. None of the wavenumbers are Wood frequencies. In the case $N=39$ we used $M=256$ resulting in a discrete DDM linear system with $20480$ unknowns; in the case $N=79$ we used $M=512$ resulting in a discrete DDM linear system with $81920$ unknowns. The large sized DDM systems are solved via the Schur complement elimination procedure.\label{comp8}}
\end{center}
\end{table}

{\em Inclusions in periodic layers}. In the last part of the numerical results section we present numerical experiments concerning periodic layers with embedded perfectly reflecting inclusions as presented in Figure~\ref{fig:inclusion}. Specifically, we consider perfectly reflecting inclusions $D$ whose boundary $\partial D$ is a smooth closed curve given in parametric form $\partial D:=\{(x_1(t),x_2(t)): x_1(t)=3.3 + r(t)\cos{t}, x_2(t)=-1+r(t)\sin{t}, r(t)=0.8 +0.4\cos{3t}, 0\leq t\leq 2\pi\}$. The inclusions $D$ are embedded periodically in a layered structure whose top boundary is explicitly given by either $\Gamma_t:=\{(x_1,F_0(x_1)): F_0(x_1)=0.6+H/2\cos{x_1}, H=0.6\}$ or the flat interface $\Gamma_t:=\{(x_1,0.6)\}$, and the bottom interface is given by shifting the top interface 3 units down the $x_2$ axis. The first set of numerical results presented in Table~\ref{comp9} concerns such configurations in the case when all three wavenumbers $k_0,k_1$ and $k_2$ are Wood frequencies.

\begin{figure}
\centering
\includegraphics[scale=0.35]{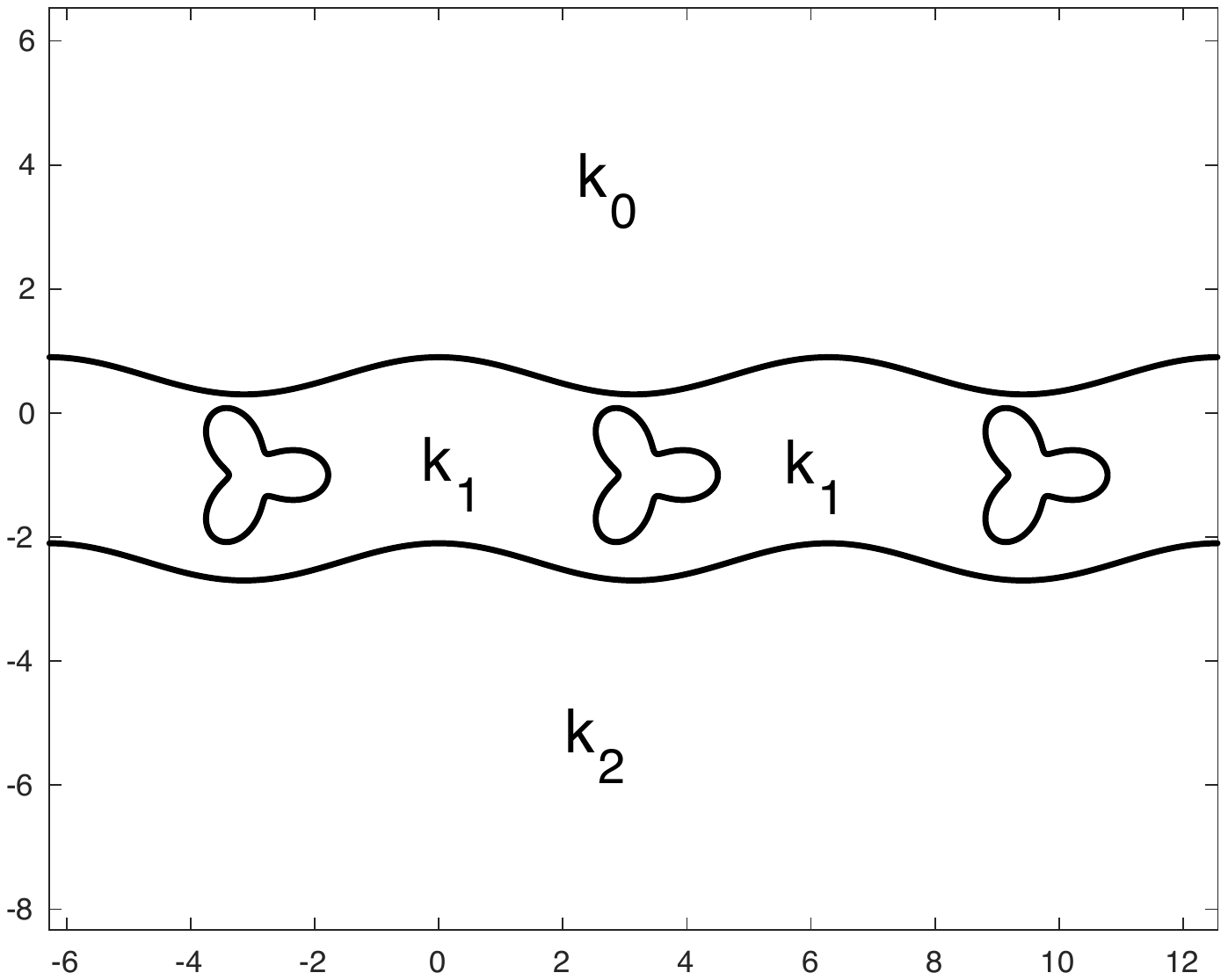}\includegraphics[scale=0.35]{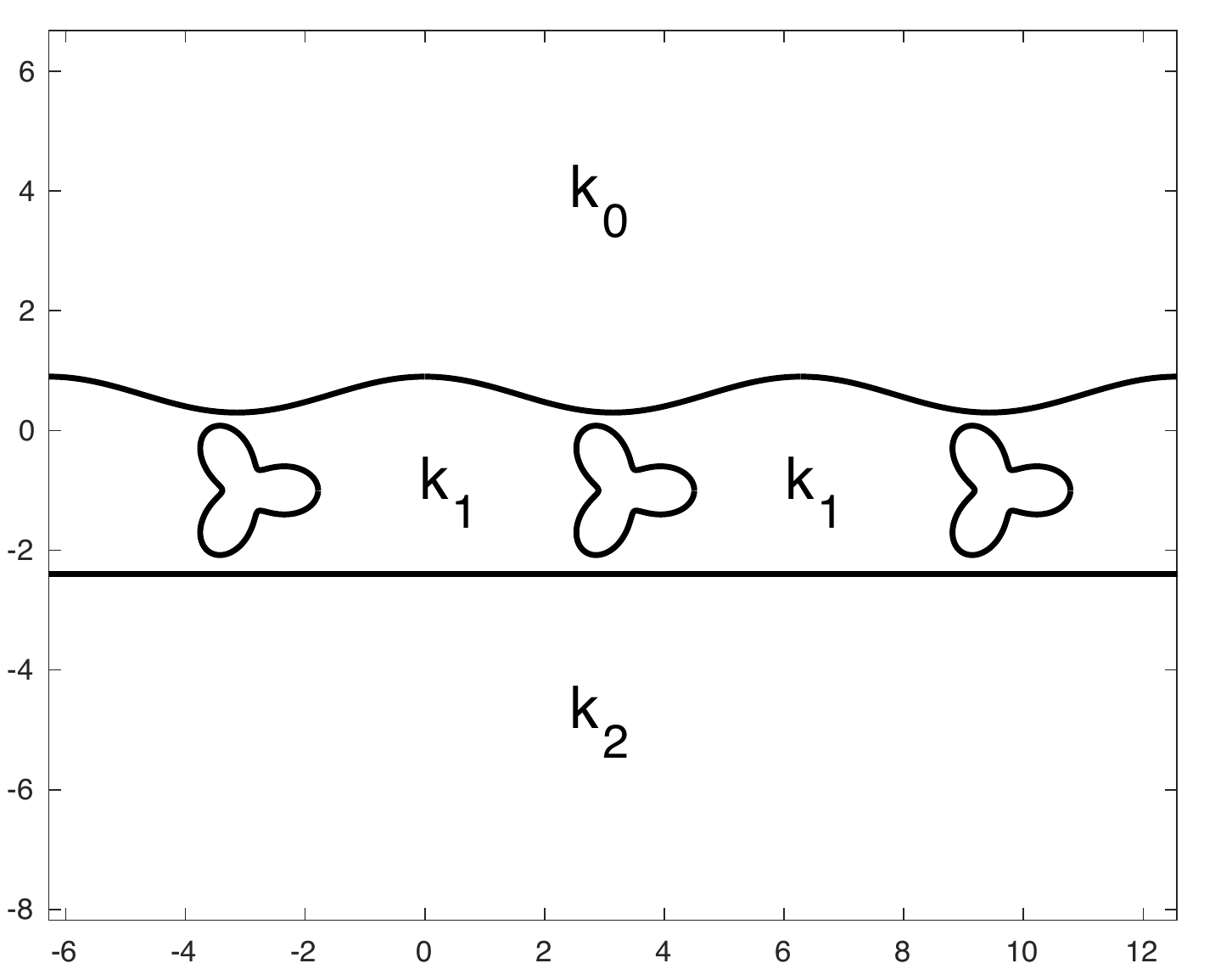}\includegraphics[scale=0.35]{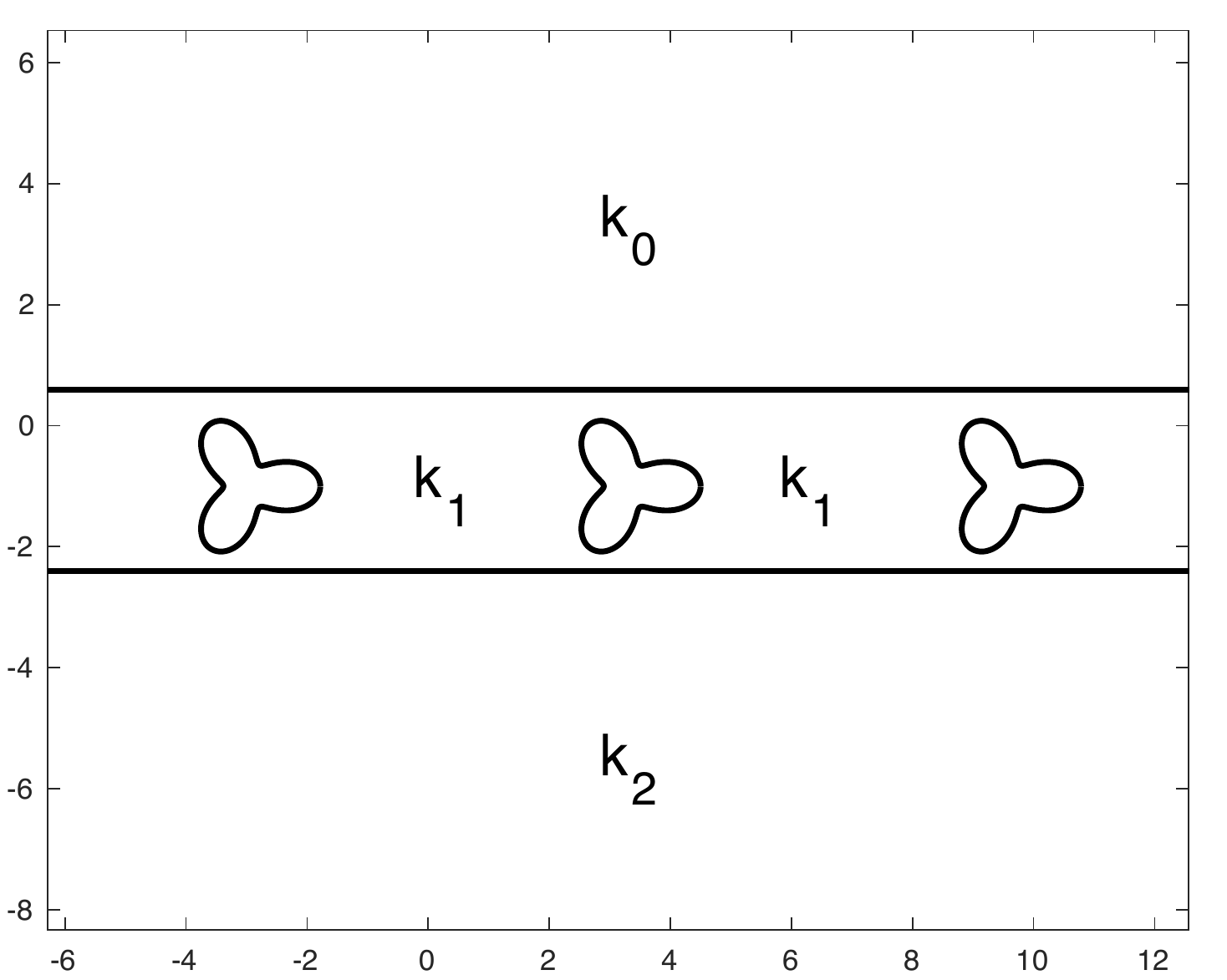}
\caption{Typical inclusions in a layered medium.}
\label{fig:inclusion}
\end{figure}

\begin{table}
  \begin{center}
    \resizebox{!}{1.9cm}
{   
\begin{tabular}{|c|c|c|c|c|c|c|c|c|}
\hline
\multicolumn{3}{|c|} {$G_{k_\ell,h_\ell}^{q,3,A}$} & \multicolumn{3}{c|} {$G_{k_\ell,h_\ell}^{q,3,A}$} & \multicolumn{3}{c|} {$G_{k_\ell,h_\ell}^{q,3,A}$}\\
\cline{1-9}
A & $\varepsilon_{en}$ & $\varepsilon_{1}$ & A & $\varepsilon_{en}$ & $\varepsilon_{1}$ & A & $\varepsilon_{en}$ & $\varepsilon_{1}$\\
\hline
40 & 1.8 $\times$ $10^{-2}$ & 2.8 $\times$ $10^{-2}$ & 40 & 1.1 $\times$ $10^{-2}$ & 1.8 $\times$ $10^{-1}$ & 40 & 7.8 $\times$ $10^{-3}$ & 2.2 $\times$ $10^{-2}$\\
80 & 4.3 $\times$ $10^{-3}$ & 4.8 $\times$ $10^{-3}$ & 80 & 2.6 $\times$ $10^{-3}$ & 3.0 $\times$ $10^{-3}$ & 80 &  2.4 $\times$ $10^{-3}$ & 3.7 $\times$ $10^{-3}$\\
120 & 1.6 $\times$ $10^{-3}$ & 1.4 $\times$ $10^{-3}$ & 120 & 3.2 $\times$ $10^{-4}$ & 8.0 $\times$ $10^{-4}$ & 120 & 2.6 $\times$ $10^{-4}$ & 1.1 $\times$ $10^{-3}$\\
\hline
\hline
40 & 1.5 $\times$ $10^{-2}$ & 4.1 $\times$ $10^{-3}$ & 40 & 1.8 $\times$ $10^{-2}$ & 6.0 $\times$ $10^{-3}$ & 40 & 4.1 $\times$ $10^{-3}$ & 7.1 $\times$ $10^{-3}$\\
80 & 5.3 $\times$ $10^{-3}$ & 6.2 $\times$ $10^{-4}$ & 80 & 5.7 $\times$ $10^{-3}$ & 9.4 $\times$ $10^{-4}$ & 80 &  2.1 $\times$ $10^{-3}$ & 1.2 $\times$ $10^{-3}$\\
120 & 1.0 $\times$ $10^{-3}$ & 2.1 $\times$ $10^{-4}$ & 120 & 1.1 $\times$ $10^{-3}$ & 2.1 $\times$ $10^{-4}$ & 120 & 2.7 $\times$ $10^{-4}$ & 3.7 $\times$ $10^{-4}$\\
\hline
\end{tabular}
}
\caption{Convergence of the DDM algorithm in the case of transmission problems for periodic configurations with perfectly reflecting inclusions depicted in Figure~\ref{fig:inclusion} with $k_\ell=\ell+1,0\leq l\leq 2$ in the top panel and $k_\ell=\ell+4,0\leq 2$ in the bottom panel, $M=64$ in each case. The wavenumbers were chosen to be Wood frequencies, and the shifts were chosen to be $h_0=0.4=-h_2$, and $h_1=3.3$. Each panel in the table (left, center, right) corresponds to the analogue periodic configuration in Figure~\ref{inclusions}. We used reference solutions for which the conservation of energy balance was of the order $10^{-5}$.\label{comp9}}
\end{center}
\end{table}

Finally, we conclude with a numerical experiment in Table~\ref{comp10} concerning scattering by a periodic arrays of perfectly reflecting cylinders $D$ described above at Wood frequencies. We treat this case via fictitious periodic layers bounded by flat interfaces as in the right panel of Figure~\ref{fig:inclusion}, for which we apply the DDM transmission algorithm with $k_0=k_1=k_2$.  The DDM approach for the solution of scattering by array of cylinders at Wood frequencies requires discretization of fictitious boundaries and as such is not as computationally efficient as alternative approaches \notesps{that rely on Sherman-Morrison formula}~\cite{bruno2017rapidly}. Nevertheless, we believe that the DDM approach is more straightforward and more modular in the sense that it consists of several black-box solvers that can be easily assembled to treat complex periodic-layered cases.   

\begin{table}
  \begin{center}
    \resizebox{!}{1.2cm}
{   
\begin{tabular}{|c|c|c|c|c|c|c|c|c|}
\hline
\multicolumn{3}{|c|} {$G_{1,h_\ell}^{q,3,A}$} & \multicolumn{3}{c|} {$G_{2,h_\ell}^{q,3,A}$} & \multicolumn{3}{c|} {$G_{4,h_\ell}^{q,3,A}$}\\
\cline{1-9}
A & $\varepsilon_{en}$ & $\varepsilon_{1}$ & A & $\varepsilon_{en}$ & $\varepsilon_{1}$ & A & $\varepsilon_{en}$ & $\varepsilon_{1}$\\
\hline
100 & 6.4 $\times$ $10^{-5}$ & 7.8 $\times$ $10^{-5}$ & 100 & 7.9 $\times$ $10^{-3}$ & 2.1 $\times$ $10^{-3}$ & 100 & 8.2 $\times$ $10^{-3}$ & 1.2 $\times$ $10^{-3}$\\
200 & 1.0 $\times$ $10^{-5}$ & 1.2 $\times$ $10^{-5}$ & 200 & 1.4 $\times$ $10^{-3}$ & 3.5 $\times$ $10^{-4}$ & 200 &  1.1 $\times$ $10^{-3}$ & 1.5 $\times$ $10^{-4}$\\
400 & 6.7 $\times$ $10^{-7}$ & 1.6 $\times$ $10^{-6}$ & 400 & 2.4 $\times$ $10^{-4}$ & 4.2 $\times$ $10^{-5}$ & 400 & 1.5 $\times$ $10^{-4}$ & 1.5 $\times$ $10^{-5}$\\
\hline
\end{tabular}
}
\caption{Scattering by an array of cylinders as presented in the right panel of Figure~\ref{fig:inclusion} with $k_0=k_1=k_2=k$, $M=64$ in each case. The wavenumbers were chosen to be Wood frequencies, and the shifts were selected to be $h_0=0.4=-h_2$, and $h_1=3.3$. We used reference solutions for which the conservation of energy balance was of the order $10^{-8}$ for $k=1$ and respectively $10^{-5}$ for $k=2,4$.\label{comp10}}
\end{center}
\end{table}

{\em 3D results.} We start our presentation of three-dimensional results with the case of scalar Helmholtz transmission problems featuring one interface of material discontinuity given by the doubly periodic grating surface $x_3=f(x_1,x_2)=\frac{1}{2}\cos(2 \pi x_1)\cos(2 \pi x_2)$.  For all the numerical experiments presented here, a single patch was used to represent the doubly periodic surfaces. We illustrate in Table~\ref{errors3} the high-order convergence achieved by our DDM solvers in the case of transmission problems involving Wood frequencies in both semi-infinite domains. For the grating considered, under normal incidence, the first three Wood frequencies occur at $2\pi$, $2\sqrt{2}\pi$, and $4\pi$ respectively.

\begin{table}
\begin{center}
\begin{tabular}{|c|c|c|c|c|c|}
\hline
$k_0$ & $k_1$ & $A$ & $\varepsilon_{en}$ & $\varepsilon_1$ \\
\hline
$2 \pi$ & $2\sqrt{2}\pi$ & 20 & 4.5 $\times$ $10^{-3}$& 1.5 $\times$ $10^{-3}$ \\
$2 \pi$ & $2\sqrt{2}\pi$ & 30 & 1.5 $\times$ $10^{-3}$& 3.1 $\times$ $10^{-4}$ \\
$2 \pi$ & $2\sqrt{2}\pi$ & 40 & 2.4 $\times$ $10^{-5}$& 2.1 $\times$ $10^{-5}$ \\
\hline
\hline
$2 \pi$ & $4\pi$ & 20 & 7.4 $\times$ $10^{-3}$& 3.4 $\times$ $10^{-3}$ \\
$2 \pi$ & $4\pi$ & 30 & 5.7 $\times$ $10^{-4}$& 3.1 $\times$ $10^{-4}$ \\
$2 \pi$ & $4\pi$ & 40 & 4.5 $\times$ $10^{-5}$& 3.6 $\times$ $10^{-5}$ \\
\hline
\end{tabular}
\caption{\label{errors3}  Convergence of the DDM transmission solver in the case of one interface of material discontinuity $\Gamma_0$ given by the grating profile $x_3=\frac{1}{2}\cos(2 \pi x_1)\cos(2 \pi x_2)$, normal incidence, and various wavenumbers that are both Wood frequencies. In both cases we used shifted quasiperiodic Green functions $G^{q,3,A}_{k_\ell,h}$ with $h=1.4$ and $M=1024$. The reference solutions were computed using $A=100$ with corresponding $\varepsilon_{en}$ of the order $10^{-6}$.}
\end{center}
\end{table}

We continue in Table~\ref{errors4} with an illustration of the accuracy achieved by our DDM solvers in the case of three layers separated by two doubly periodic gratings and wavenumber configurations that involve Wood frequencies. Finally, we conclude with an illustration in Table~\ref{comp83d} of the ability of the Schur complement DDM solvers to handle very large numbers of layers in three dimensions that require large discretizations; for instance, the largest problem considered in Table~\ref{comp83d} involves a periodic layered configuration consisting of 80 doubly periodic interfaces of material discontinuity, spanning about 160 wavelengths, whose DDM discretization required 163840 unknowns.  \notesps{We mention that the condition numbers of the matrices $\mathcal{D}_{j,M}$ (see Section~\ref{nystrom}) that need be inverted in the Schur complement solution of the problems presented in Table~\ref{comp83d} belong to the interval $[10^2,3.3\times 10^{3}]$}.

\begin{table}
\begin{center}
\begin{tabular}{|c|c|c|c|c|c|c|}
\hline
$k_0$ & $k_1$ & $k_2$ & $A$ & $\varepsilon_{en}$ & $\varepsilon_1$ \\
\hline
$1$ & $2$& $2\pi$ (W) & 20 & 1.2 $\times$ $10^{-1}$ & 5.0 $\times$ $10^{-2}$ \\
$1$ & $2$& $2\pi$ (W) & 40 & 2.7 $\times$ $10^{-3}$ & 1.7 $\times$ $10^{-2}$ \\
$1$ & $2$& $2\pi$  (W) & 60 & 4.2 $\times$ $10^{-4}$ & 6.4 $\times$ $10^{-4}$ \\
\hline
\hline
$1$ & $2\pi$ (W) & $2$ & 20 & 1.1 $\times$ $10^{-1}$ & 7.3 $\times$ $10^{-2}$ \\
$1$ & $2\pi$ (W) & $2$ & 40 & 7.4 $\times$ $10^{-3}$ & 8.7 $\times$ $10^{-3}$ \\
$1$ & $2\pi$ (W) & $2$  & 60 & 1.4 $\times$ $10^{-3}$& 6.4 $\times$ $10^{-4}$ \\
\hline
\hline
$2 \pi$ (W) & $2\sqrt{2}\pi$ (W) & $4\pi$ (W) & 20 & 5.3 $\times$ $10^{-2}$& 2.5 $\times$ $10^{-2}$ \\
$2 \pi$ (W) & $2\sqrt{2}\pi$ (W) & $4\pi$  (W) & 40 & 7.7 $\times$ $10^{-3}$& 5.7 $\times$ $10^{-3}$ \\
$2 \pi$ (W) & $2\sqrt{2}\pi$ (W) & $4\pi$ (W) & 60 & 1.8 $\times$ $10^{-3}$& 6.3 $\times$ $10^{-4}$ \\
\hline
\end{tabular}
\caption{\label{errors4}  Convergence of the DDM transmission solver in the case of three layers ($N=2$) separated by grating profiles grating profile $F_0(x_1,x_2)=\frac{1}{2}\cos(2 \pi x_1)\cos(2 \pi x_2)$ and $F_1(x_1,x_2)=F_0(x_1,x_2)-1.3$ under normal incidence. In both cases we used shifted quasiperiodic Green functions $G^{q,3,A}_{k_0,h}$ and $G^{q,3,A}_{k_1,-h}$ with $h=1.4$ and $M=1024$. The reference solutions were computed using $A=100$ with corresponding $\varepsilon_{en}$ of the order $10^{-6}$.}
\end{center}
\end{table}

\begin{table}
   \begin{center}
\begin{tabular}{|c|c|c|c|}
\hline
\multicolumn{2}{|c|} {$G_{k_\ell}^{q,A}, N=19$} & \multicolumn{2}{|c|} {$G_{k_\ell}^{q,A}, N=79$} \\
\cline{1-4}
A & $\varepsilon_{en}$ & A & $\varepsilon_{en}$ \\
\hline
40 & 2.1 $\times$ $10^{-2}$  & 40 & 1.1 $\times$ $10^{-1}$\\
60 & 3.9 $\times$ $10^{-3}$ & 60 & 2.8 $\times$ $10^{-2}$\\
\hline
\end{tabular}
\caption{Convergence of the DDM transmission solver for configuration consisting of $N+2$ layers for various values of $N$, where the interfaces $\Gamma_\ell,0\leq \ell\leq N$ are given by grating profiles $F_\ell(x_1,x_2)=-\ell L+1/2\cos(2\pi x_1)\cos(2\pi x_2), L=1.3, 0\leq \ell\leq N$, under normal incidence, with $k_\ell$ drawn randomly from the interval $[1,25]$, and discretization size $M=1032$  resulting in discrete DDM linear systems with $40960$ unknowns in the case $N=19$  and respectively $163840$ unknowns in the case $N=79$. The ensuing DDM systems are solved via the Schur complement elimination procedure.\label{comp83d}}
\end{center}
\end{table}

While the grating profiles considered in this paper are relatively simple, qualitatively similar results can be obtained for more geometrically complex profiles. We mention that extensions to grating profiles that involve corners is straightforward in two dimensions; graded meshes and weighted versions of RtR maps are needed to treat those cases~\cite{jerez2017multitrace}. Extensions to three-dimensional gratings with edges and corners can be done by applying existing technology~\cite{jerez2017multitrace}.  The results presented in this section were produced on a MacBookPro with a 2.7 GHz Intel processor and 8Gb of RAM based on a MATLAB implementation of the DDM algorithm. We did not strive to optimize the code in order to harness the best sequential computational performance; this is certainly possible, and it has been done in~\cite{bruno2016superalgebraically,bruno2017shifted} via fast methods based on equivalent sources. Instead, we wanted to illustrate that, within a DDM approach, the use of windowed Green function method combined with shifted Green functions leads to a computational method for scattering by periodic layered media that is accurate and robust at all frequencies, including the challenging Wood frequencies.
  
\section{Conclusions}\label{conclu}

We presented analysis and numerical experiments concerning  boundary-integral operators-based DDM for two and three-dimensional periodic layered media scalar scattering problems. We have shown that the RtR maps that are needed by DDM can be computed in a robust manner at all frequencies, including Wood frequencies. The Wood frequencies configurations were treated via boundary-integral operators that incorporate shifted quasi-periodic Green functions that converge at Wood frequencies. The tridiagonal DDM linear system associated with transmission problems in periodic layered media was solved via recursive Schur complements resulting in a computational cost that is linear in the number of layers. Extensions to full three-dimensional electromagnetic configurations are straightforward. We are currently investigating the design of DDM with quasi-optimal transmission conditions for the solution of transmission problems in periodic layered media.

\section*{Acknowledgments}
Stephen Shipman acknowledges support from NSF through contract DMS-0807325. Catalin Turc acknowledges support from NSF through contract DMS-1614270. Stephanos Venakides acknowledges support from NSF through contract DMS-1211638.

\section{Appendix}\label{wp_proof}

\begin{theorem}\label{thm:uniqueness1}
  Under the assumptions that (1) the wavenumbers $k_j$ are such that $0\leq k_j\leq k_{j+1}$ for all $0\leq j\leq N$ and (2) the coefficients $\gamma_j=1$ for all $0\leq j\leq N+1$, the system of Helmholtz equations~\eqref{system_t} has a unique solution when the functions $F_j$ are $C^2$.
\end{theorem}
\begin{proof}
  Clearly, the uniqueness of solutions amounts to showing that when the incident field is zero, the only solution of the transmission equations~\eqref{system_t} is the trivial solution. The key ingredient in the proof is the application of Green's identities. Let us choose $h>\max{F_0}$ and define the domain $\Omega_{0,h}^{per}:=\{(x_1,x_2)\in\Omega_0^{per}: F_0(x_1)\leq x_2\leq h\}$. A simple application of Green's identities leads to 
\begin{equation*}
  \int_{\Omega_{0,h}^{per}}(|\nabla u_0|^2-k_0^2|u_0|^2)dx=\int_{\Gamma_0}\partial_{n_0}u_0\ \overline{u_0}\ ds + \int_{\Gamma_{0,h}}\partial_{x_2}u_0\ \overline{u_0}\ dx_1
\end{equation*}
where $\Gamma_{0,h}:=\{(x_1,x_2): 0\leq x_1\leq d,\ x_2=h\}$. Note that the integrals over the vertical lines vanish due to the quasi-periodicity of the field. Taking into account the fact that $u_0$ is radiating, we can express $u_0$ on the line segment $\Gamma_{0,h}$ in terms of the following Rayleigh series
\[
u_0(x_1,h)=\sum_{r\in\mathbb{Z}} C_r^{+}e^{i\alpha_r x_1+i\beta_{0,r} h}
\]
 from which it follows immediately that
\[
\lim_{h\to\infty}\int_{\Gamma_{0,h}}\partial_{x_2}u_0\ \overline{u_0}\ dx_1=id\sum_{r\in\mathbb{Z},\ \beta_{0,r}>0}\beta_{0,r}|C_r^{+}|^2.
\]
Hence, we get
\begin{equation*}
  \int_{\Omega_{0}^{per}}(|\nabla u_0|^2-k_0^2|u_0|^2)dx=\int_{\Gamma_0}\partial_{n_0}u_0\ \overline{u_0}\ ds + id\sum_{r\in\mathbb{Z},\ \beta_{0,r}>0}\beta_{0,r}|C_r^{+}|^2.
\end{equation*}
Taking the imaginary part of the equation above we arrive at
\begin{equation}\label{eq:first_identity}
  \Im\int_{\Gamma_0}(\partial_{n_0}u_0)\ \overline{u_0}\ ds =-d\sum_{r\in\mathbb{Z},\ \beta_{0,r}>0}\beta_{0,r}|C_r^{+}|^2.
\end{equation}
On the other hand, application of the Green identities in the layers $\Omega_j^{per},1\leq j\leq N$ that have a finite width in the $x_2$ leads to
\begin{equation*}
  \int_{\Omega_{j}^{per}}(|\nabla u_j|^2-k_j^2|u_j|^2)dx=\int_{\Gamma_{j-1}}\partial_{n_j}u_j\ \overline{u_j}\ ds + \int_{\Gamma_{j}}\partial_{n_j}u_j\ \overline{u_j}\ ds.
\end{equation*}
Taking the imaginary part in the equation above we obtain
\begin{equation}\label{eq:mid_layer_j}
  \Im\int_{\Gamma_{j-1}}(\partial_{n_j}u_j)\ \overline{u_j}\ ds = -\Im \int_{\Gamma_{j}}(\partial_{n_j}u_j)\ \overline{u_j}\ ds.
\end{equation}
Applying the same arguments that led to the derivation of equation~\eqref{eq:first_identity} in the case of the semi-infinite layer $\Omega_{N+1}$ we obtain
\begin{equation}\label{eq:last_identity}
  \Im \int_{\Gamma_{N}}(\partial_{n_{N+1}}u_{N+1})\ \overline{u_{N+1}}\ ds =-d\sum_{r\in\mathbb{Z},\ \beta_{N+1,r}>0}\beta_{N+1,r}|C_r^{-}|^2.
\end{equation}
Adding the left-hand sides as well as the right  hand sides of equations~\eqref{eq:first_identity},~\eqref{eq:mid_layer_j}, and~\eqref{eq:last_identity} and taking into account the continuity conditions in the transmission system~\eqref{system_t} we obtain
\[
\sum_{r\in\mathbb{Z},\ \beta_{0,r}>0}\beta_{0,r}|C_r^{+}|^2+ \sum_{r\in\mathbb{Z},\ \beta_{N+1,r}>0}\beta_{N+1,r}|C_r^{-}|^2=0.
\]
The last equation implies that the Rayleigh coefficients of the propagating modes corresponding to $u_0$ and $u_{N+1}$ are all equal to zero, that is  $C_r^{+}=0$ for all $r$ such that $\beta_{0,r}>0$ as well as $C_r^{-}=0$ for all $r$ such that $\beta_{N+1,r}>0$. We have then
\[
u_0(x_1,h)=\sum_{r\in\mathbb{Z},\beta_{0,r}=0} C_r^{+}e^{i\alpha_r x_1+i\beta_{0,r} h} + \sum_{r\in\mathbb{Z},\Im{\beta_{0,r}}>0} C_r^{+}e^{i\alpha_r x_1+i\beta_{0,r} h}.
\]
We define $v_0:=\partial_{x_2}u_0$ in the domain $\Omega_0$ and we apply Green's third identity to the functions $u_0$ and $\overline{v_0}$ in the domain $\Omega_{0,h}^{per}$ and take $h\to\infty$ to obtain
\begin{equation}\label{eq:1}
\int_{\Gamma_0}\partial_{n_0}u_0\ \overline{v_0}\ ds=\int_{\Gamma_0}u_0\ \partial_{n_0}\overline{v_0}\ ds.
\end{equation}
We similarly define $v_j:=\partial_{x_2}u_j$ in the domains $\Omega_j$ for $1\leq j\leq N$ and we obtain in a similar manner
\begin{equation}\label{eq:2}
  \int_{\Gamma_{j-1}} \partial_{n_j}u_j\ \overline{v_j}\ ds+\int_{\Gamma_j}\partial_{n_j}u_j\ \overline{v_j}\ ds=\int_{\Gamma_{j-1}} u_j\ \partial_{n_j}\overline{v_j}\ ds+\int_{\Gamma_j}u_j\ \partial_{n_j}\overline{v_j}\ ds,\ 1\leq j\leq N
\end{equation}
as well as
\begin{equation}\label{eq:3}
\int_{\Gamma_{N+1}}\partial_{n_{N+1}}u_{N+1}\ \overline{v_{N+1}}\ ds=\int_{\Gamma_{N+1}} u_{N+1}\ \partial_{n_{N+1}}\overline{v_{N+1}}\ ds.
\end{equation}
Now, using the continuity of the normal derivatives and the tangential derivatives of the fields $u_j$ across interfaces $\Gamma_j$, it follows immediately that the quantities $v_j$ are continuous across the interfaces $\Gamma_j$. Also, as shown in~\cite{Petit}, we have that
\begin{equation}\label{eq:jump_n}
\partial_{n_j}v_j+\partial_{n_{j+1}}v_{j+1}=(k_{j+1}^2-k_j^2)n_{j,x_2}u_j\quad {\rm on}\ \Gamma_j, 1\leq j\leq N,
\end{equation}
where $n_{j,x_2}$ denotes the component of the normal $n_j$ along the $x_2$ axis. Taking these last two facts into account, we add the left-hand sides and right hand sides of equations~\eqref{eq:1},~\eqref{eq:2}, and~\eqref{eq:3} and we get
\begin{equation}\label{eq:4}
  \sum_{j=0}^N\int_{\Gamma_j}(\partial_{n_j}u_j+\partial_{n_{j+1}}u_{j+1})\overline{v_j}\ ds=\sum_{j=0}^N\int_{\Gamma_j}u_j(\overline{\partial_{n_j}v_j} + \overline{\partial_{n_{j+1}}v_{j+1}})\ ds
\end{equation}
Now, given that $\partial_{n_j}u_j+\partial_{n_{j+1}}u_{j+1}=0$ on $\Gamma_j$ for $0\leq j\leq N$, and taking into account the continuity condition in equation~\eqref{eq:jump_n}, we obtain
\begin{equation}\label{eq:central_identity}
\sum_{j=0}^N(k_{j+1}^2-k_j^2)\int_{\Gamma_j}n_{j,x_2}|u_j|^2ds=0.
\end{equation}
Now, given the assumption (1) and the fact the normals $n_j$ are chosen to point to the exterior of the domains $\Omega_j$ and hence $n_{j,x_2}<0$ for all $0\leq j\leq N$, it follows in particular that $u_0=0$ on $\Gamma_0$. In the light of this fact, we revisit formula~\eqref{eq:1} and we get that
\[
\int_{\Gamma_0}\partial_{n_0}u_0\ \overline{v_0}\ ds =0.
\]
Given that $u_0=0$ on $\Gamma_0$, it follows that its tangential derivative is also equal to zero on $\Gamma_0$. Denoting by $w_0:=\partial_{x_1}u_0$ in $\Omega_0$, the latter fact translates into $n_{0,x_2}w_0=n_{0,x_1}v_0$ on $\Gamma_0$. Since $\partial_{n_0}u_0\ \overline{v_0}=n_{0,x_1}w_0\ \overline{v_0}+n_{0,x_2}|v_0|^2=n_{0,x_2}|w_0|^2+n_{0,x_2}|v_0|^2$, we get 
\[
\int_{\Gamma_0}n_{0,x_2}(|w_0|^2+|v_0|^2)ds=0.
\]
Hence, $u_0=0,\ \partial_{n_0}u_0=0$ on $\Gamma_0$, which implies that $u_0=0$ in $\Omega_0$ by Holmgren's theorem. Now the use of continuity conditions across interfaces $\Gamma_j$ for $1\leq j\leq N$ and Holmgren's theorem leads to the conclusion of the theorem.
\end{proof}

\section{Appendix}\label{proof_Wood}
We devote this Appendix to proving the following result.

{\bf Theorem 3.4} \emph{
  Under the assumption that $F_0$ is $C^2$ and that $k_0$ is a Wood frequency, the operator
  \[
  \mathcal{A}_{0,h}:=\frac{1}{2}I+(K_{\Gamma_0,k_0,h}^{q,j})^\top-Z_0S_{\Gamma_0,k_0,h}^{q,j},\ j\geq 1, \mathcal{A}_{0,h}:L^2_{per}(\Gamma_0)\to l^2_{per}(\Gamma_0)
  \]
  is invertible with continuous inverse for all but a discrete set of values of the shift $h>0$.
}

\begin{proof}
  Since for a given $\mathbf{y}\in\Gamma_0$, the kernels $G_{k_0,h}^{q,j}(\cdot-\mathbf{y})$ have the same singularity on $\Gamma_0$ (that is at $\mathbf{x}=\mathbf{y}$) as the kernels $G_{k_0}^q(\cdot-\mathbf{y})$, then both operators $(K_{\Gamma_0,k_0,h}^{q,j})^\top:L^2_{per}(\Gamma_0)\to L^2_{per}(\Gamma_0)$ and $S_{\Gamma_0,k_0,h}^{q,j}:L^2_{per}(\Gamma_0)\to L^2_{per}(\Gamma_0)$ are compact. Thus, the conclusion of the Theorem follows once we establish the injectivity of the operator $\mathcal{A}_{0,h}$. Let $\varphi\in Ker(\mathcal{A}_{0,h})$ and define
  \[
w_0:=SL^{q,j}_{k_0,h}\varphi\quad{\rm in}\ \mathbb{R}^2\setminus\Gamma_0.
\]
The function $w_0$ is a radiating $\alpha$-quasi-periodic solution of the Helmholtz equation in the domain $\Omega_0$ with zero Robin boundary values on $\Gamma_0$, that is
\[
\partial_{n_0}w_0-Z_0w_0=0\quad{\rm on}\quad\Gamma_0.
\]
It follows from Theorem~\ref{wp_Omega_0} that $w_0=0$ in $\Omega_0$. It is straightforward to see that the shifted function $G_{k_0,h}^{q,j}$ has the following frequency domain representation~\cite{Delourme}:
\[
G_{k_0,h}^{q,j}(x_1,x_2;y_1,y_2)=\frac{i}{2d}\sum_{r\notin U} e^{i\alpha_r(x_1-y_1)}\frac{(1-e^{i\beta_r h})^j}{\beta_r}e^{i\beta_r|x_2-y_2|}+\sum_{r\in U} c_r e^{i\alpha_r(x_1-y_1)}, 
\]
with $(x_1,x_2)\in\Omega_0,\ (y_1,y_2)\in\Omega_0$ where $U:=\{r\in\mathbb{Z}: \beta_r=0\}$; clearly $U$ is not empty since $k_0$ is a Wood frequency. Hence, the solution $w_0$ admits the following representation
\[
w_0(x_1,x_2)=\sum_{r}w_{0,r}^{+}(x_2)e^{i\alpha_r x_1},\ x_2>\max{F_0}
\]
where
\[
w_{0,r}^{+}(x_2)=\begin{cases}e^{i\beta_r x_2}\frac{(1-e^{i\beta_r h})^j}{\beta_r}b_r^{+}& r\notin U\\
c_rb_r^{+}& r\in U\end{cases}
\]
with
\[
b_r^{+}=\frac{i}{2d}\int_{\Gamma_0}e^{-i\alpha_ry_1}e^{-i\beta_ry_2}\varphi(y_1,y_2)ds(y_1,y_2).
\]
Assuming that the shift $h$ is chosen such that $1-e^{i\beta_r h}\neq 0$ for all $r\in U$, it follows immediately that $w_0=0$ in $\Omega_0$ implies that all the coefficients $b_r^{+}=0$ for all $r$.

The key insight in the proof~\cite{bruno2017three} is to introduce the following $\alpha$-quasi-periodic Green function that is defined even at Wood frequencies
\[
B^q(x_1,x_2)=\frac{i}{2d}\sum_{r\notin U}e^{i\alpha_r x_1}\frac{e^{i\beta_r|x_2|}}{\beta_r}+\frac{i}{2d}\sum_{r\in U}e^{i\alpha_r x_1}i|x_2|.
\]
The function $B^q$ is not an outgoing Green function on account of the linear term $|x_2|$. Now define
\[
v(\mathbf{x}):=\int_{\Gamma_0}B^q(\mathbf{x}-\mathbf{y})\varphi(\mathbf{y})ds(\mathbf{y}),\ \mathbf{x}\notin\Gamma_0.
\]
Then the function $v$ admits the representation
\[
v(x_1,x_2)=\sum_{r}v_{r}^{+}(x_2)e^{i\alpha_r x_1},\ x_2>\max{F_0}
\]
where
\[
v_{r}^{+}(x_2)=\begin{cases}\frac{e^{i\beta_r x_2}}{\beta_r}b_r^{+}& r\notin U\\
ix_2b_r^{+}-ib_r'& r\in U\end{cases}
\]
with
\[
  b_r'=\frac{i}{2d}\int_{\Gamma_0}e^{-i\alpha_ry_1}y_2\varphi(y_1,y_2)ds(y_1,y_2),\ r\in U.
\]
Given that we established the fact that $b_r^{+}=0$ for all $r$, we obtain $v(x_1,x_2)=\sum_{r\in U}(-ib_r')e^{-i\alpha_rx_1}$ for $x_2>\max{F_0}$. Certainly, $v(\mathbf{x})$ is a solution of the Helmholtz equation for $\mathbf{x}\notin\Gamma_0$, and $v(\mathbf{x})$ is independent of $x_2$ for $x_2>\max{F_0}$. But $v(\mathbf{x})$ is real analytic for $\mathbf{x}\notin\Gamma_0$, so it follows from analytic continuation that $v(\mathbf{x})$ is actually independent of $x_2$ everywhere in $\Omega_0$. Thus, we have
\[
v(x_1,x_2)=\sum_{r\in U}(-ib_r')e^{-i\alpha_rx_1}\quad{\rm for}\ x_2\geq F_0(x_1).
\]
We study next the behavior of the function $v(x_1,x_2)$ for $x_2<\min{F_0}$. We get immediately that $v$ also admits the representation
\[
v(x_1,x_2)=\sum_{r}v_{r}^{-}(x_2)e^{i\alpha_r x_1},\ x_2<\min{F_0}
\]
where
\[
v_{r}^{-}(x_2)=\begin{cases}\frac{e^{-i\beta_r x_2}}{\beta_r}b_r^{-}& r\notin U\\
-ix_2b_r^{-}+ib_r'& r\in U\end{cases}
\]
with
\[
b_r^{-}=\frac{i}{2d}\int_{\Gamma_0}e^{-i\alpha_ry_1}e^{i\beta_ry_2}\varphi(y_1,y_2)ds(y_1,y_2),\ r\in\mathbb{Z}.
\]
Comparing the definitions of the coefficients $b_r^{+}$ and $b_r^{-}$ we see that they differ in general for $r\notin U$. However, and most importantly, we have that $b_r^{-}=b_{r}^{+}$ for $r\in U$ since $\beta_r=0$ for $r\in U$. In conclusion, we have
\[
v_{r}^{-}(x_2)=\begin{cases}\frac{e^{-i\beta_r x_2}}{\beta_r}b_r^{-}& r\notin U\\
ib_r'& r\in U.\end{cases}
\]
The very last fact we established implies that $v$ is actually a radiating $\alpha$-quasi-periodic solution of the Helmholtz equation in the domain $\Omega_0^{-}:\{(x_1,x_2):x_2\leq F_0(x_1)\}$. In the last step of the proof we define
\[
\tilde{v}(x_1,x_2):=v(x_1,x_2)-\sum_{r\in U}(-ib_r')e^{-i\alpha_rx_1},\quad{\rm for} (x_1,x_2)\in\mathbb{R}^2
\]
which is a $\alpha$-quasi-periodic solution of the Helmholtz equation that satisfies the radiation condition as $x_2\to\infty$ as well as $x_2\to-\infty$, while vanishing in $\Omega_0$. Clearly, $\tilde{v}$ is a radiating $\alpha$-quasi-periodic solution of the Helmholtz equation in the domain $\Omega_0^{-}$ that vanishes on $\Gamma_0$. Consequently, $\tilde{v}=0$ in $\Omega_0$~\cite{Petit}. Using the jump conditions of the normal derivative of the normal derivative of the single-layer potentials, we get that $\varphi=0$ on $\Gamma_0$ which concludes the proof of the theorem. 
  \end{proof}
  
\bibliography{biblioLayer}

\end{document}